\documentclass[12pt]{article}
\title{Translating between NIP integral domains and topological fields}
\author{Will Johnson}

\usepackage{amsmath, amssymb, amsthm}    	% need for subequations
\usepackage{fullpage} 	% without this, will have wide math-article-paper margins
\usepackage{amscd}
\usepackage{hyperref}
\usepackage[all]{xy}
\usepackage[T1]{fontenc}
\usepackage{lmodern}
\usepackage{centernot}
\usepackage{enumitem}
\usepackage{titlefoot}

\DeclareMathOperator*{\ind}{\raise0.2ex\hbox{\ooalign{\hidewidth$\vert$\hidewidth\cr\raise-0.9ex\hbox{$\smile$}}}}

\newcommand{\ACVF}{\mathrm{ACVF}}

\newcommand{\ba}{{\bar{a}}}
\newcommand{\bb}{{\bar{b}}}

\newcommand{\bx}{{\bar{x}}}
\newcommand{\by}{{\bar{y}}}

\newcommand{\Frac}{\operatorname{Frac}}

\newcommand{\br}{\operatorname{br}}

\newcommand{\characteristic}{\operatorname{char}}

\newcommand{\dpr}{\operatorname{dp-rk}}

\newtheorem{theorem}{Theorem}[section] % numbered like the section

\newtheorem{lemma}[theorem]{Lemma}

\newtheorem{corollary}[theorem]{Corollary}
\newtheorem{fact}[theorem]{Fact}

\newtheorem{conjecture}[theorem]{Conjecture}
\newtheorem{unlikely}[theorem]{Unlikely Conjecture}

\newtheorem{proposition}[theorem]{Proposition}
\newtheorem{proposition-eh}[theorem]{Proposition(?)}
\newtheorem*{theorem-star}{Theorem}
\newtheorem*{conjecture-star}{Conjecture}
\newtheorem*{fact-star}{Fact}
\newtheorem*{lemma-star}{Lemma}
\newtheorem{claim}[theorem]{Claim}

\theoremstyle{definition}
\newtheorem{definition}[theorem]{Definition}
\newtheorem{example}[theorem]{Example}

\newtheorem{remark}[theorem]{Remark}
\newtheorem{goal}[theorem]{Goal}

\theoremstyle{remark}

\newtheorem*{acknowledgment}{Acknowledgments}

\newcommand{\Aa}{\mathbb{A}}
\newcommand{\Qq}{\mathbb{Q}}

\newcommand{\Rr}{\mathbb{R}}

\newcommand{\Nn}{\mathbb{N}}
\newcommand{\Cc}{\mathbb{C}}

\newcommand{\Ff}{\mathbb{F}}
\newcommand{\Pp}{\mathbb{P}}

\newcommand{\Oo}{\mathcal{O}}
\newcommand{\mm}{\mathfrak{m}}
\newcommand{\pp}{\mathfrak{p}}
\newcommand{\qq}{\mathfrak{q}}
\newcommand{\rr}{\mathfrak{r}}

\newenvironment{claimproof}[1][\proofname]
               {
                 \proof[#1]
                 
               }
               {
                 \endproof
               }

\begin{document}

\maketitle\unmarkedfntext{
  \emph{2020 Mathematical Subject Classification}: 03C60, 12L12.

  \emph{Key words and phrases}: NIP rings, definable field topologies
  }

\begin{abstract}
  We prove that definable ring topologies on NIP fields are closely
  connected to NIP integral domains.  More precisely, we show that up to elementary equivalence, any NIP topological field arises from an NIP integral domain.  As an application, we prove several
  results about definable ring topologies on NIP fields, including the
  following.  Let $K$ be an NIP field or expansion of a field.  Let
  $\tau$ be a definable ring topology on $K$.  Then $\tau$ is a
  \emph{field} topology, and $\tau$ is locally bounded.  If $K$ has
  characteristic $p$ or finite dp-rank, then $\tau$ is ``generalized
  t-henselian'' in the sense of Dittman, Walsberg, and Ye, meaning
  that the implicit function theorem holds for polynomials.  If $K$
  has finite dp-rank, then $\tau$ must be a topology of ``finite
  breadth'' (a $W_n$-topology).  Using these techniques, we give some
  reformulations of the conjecture that NIP local rings are henselian.
\end{abstract}

\section{Introduction}
In this paper, rings are assumed to be commutative and unital.  A
\emph{ring topology} or \emph{field topology} on a field $K$ is a
Hausdorff, non-discrete topology such that the ring operations or
field operations are continuous, respectively.  A topology on a
definable set $D$ is \emph{definable} if some uniformly definable
family of sets is a basis of opens.

The class of \emph{NIP theories} has played a major
role in contemporary model theory.  See Simon's book \cite{NIPguide} for the definition and motivation of NIP.  Many NIP theories of fields come
equipped with definable field topologies, such as the order topology
on real closed fields like $\Rr$ or the valuation topology on
$p$-adically closed fields like $\Qq_p$.
%% In
%% fact, there is evidence suggesting that \emph{every} unstable NIP
%% field admits a definable field topology
%% (Remark~\ref{conj-tops-exist}).

The goal of this note is to show a close connection between NIP
topological fields and NIP integral domains.  Recall the following
general construction of ring topologies from integral domains
\cite[Example~1.2]{PZ}:
\begin{fact} \label{pz-fact}
  Let $R$ be an integral domain with fraction field $K \ne R$.  There
  is a ring topology $\tau_R$ on $K$ such that the non-zero ideals of
  $R$ form a neighborhood basis of 0.
\end{fact}
Continuing an abuse of terminology from \cite{mana}, we call $\tau_R$
the \emph{$R$-adic topology} on $K$.  For example, when $R$ is a
non-trivial valuation ring, the $R$-adic topology is the valuation
topology.

In the case where $R$ is an NIP integral domain, we get an NIP
topological field:
\begin{fact}[{= Corollary~\ref{new-cor}}] \label{main-fact}
  Let $R$ be an NIP integral domain with fraction field $K \ne R$.
  \begin{enumerate}
  \item The $R$-adic topology $\tau_R$ is a field topology, not just a
    ring topology (division is continuous).
  \item The $R$-adic topology is definable in the pair $(K,R)$, which
    is an NIP expansion of $K$.
  \end{enumerate}
\end{fact}
Fact~\ref{main-fact} is essentially trivial modulo prior results, but
for completeness we give the proof in Section~\ref{frtft} below.

Our main theorem shows that, on some level, \emph{every} NIP
topological field arises via this construction:
\begin{theorem}[{= Theorem~\ref{main-thm2}}] \label{main-thm}
  Let $(K,+,\cdot,\ldots)$ be an NIP expansion of a field.  Let $\tau$
  be a definable ring topology on $K$.  Suppose $K$ is sufficiently
  saturated.  Then there is an externally definable subring $R
  \subseteq K$ such that $\tau$ is $\tau_R$.
\end{theorem}
Although $R$ is not necessarily definable, the expansion
$(K,+,\cdot,\ldots,R)$ is NIP by a theorem of Shelah (see
\cite[Corollary~3.24]{NIPguide} or \cite{Shelah}).  In particular, $R$
is an NIP integral domain, completing the circle.
\begin{example} \label{ex1.4}
  If $K$ is a sufficiently saturated real closed field, then the order
  topology is induced by the externally definable valuation ring $R =
  \bigcup_{n \in \Nn} [-n,n]$.
\end{example}
Note that Theorem~\ref{main-thm} is a statement about \emph{ring}
topologies on $K$, rather than \emph{field} topologies.  But
Fact~\ref{main-fact} guarantees that $\tau_R$ is a field topology,
yielding the following surprising corollary:
\begin{corollary}[{= Corollary~\ref{cor:r-to-f}}] \label{cor1.5}
  Let $K$ be an NIP expansion of a field, and let $\tau$ be a
  definable ring topology on $K$.  Then $\tau$ is a field topology
  (division is continuous).
\end{corollary}
Consequently, the reader can substitute ``ring topology'' for ``field
topology'' in the results below.

Theorem~\ref{main-thm} allows us to prove statements about NIP
topological fields by translating them into statements about NIP
integral domains.  In Subsections~\ref{subsec1}--\ref{subsec3}, we state three theorems
along these lines.
\subsection{Local boundedness} \label{subsec1}
First, recall the general notion of boundedness on topological fields
\cite[Section~2]{PZ}.  Let $K$ be a field and $\tau$ be a ring
topology on $K$.  A set $B \subseteq K$ is \emph{bounded} if for every
neighborhood $U \ni 0$, there is some non-zero $c$ with $cB \subseteq
U$.  We say that $\tau$ is \emph{locally bounded} if some neighborhood
of 0 is bounded.  The $R$-adic topology $\tau_R$ of Fact~\ref{pz-fact}
is always locally bounded \cite[Theorem~2.2]{PZ}.  Then the following
is an easy corollary of Theorem~\ref{main-thm}:
\begin{theorem}[{= Theorem~\ref{loc-bdd-thm2}}] \label{loc-bdd-thm}
  Let $\tau$ be a definable field topology on an NIP expansion of a
  field.  Then $\tau$ is locally bounded.
\end{theorem}
Actually, we will prove Theorem~\ref{loc-bdd-thm} first, and then
deduce Theorem~\ref{main-thm} as a corollary, not vice versa.

\subsection{Generalized t-henselianity} \label{subsec2}
In certain cases, NIP integral domains are known to be henselian local
rings:
\begin{fact} \label{fact1.7}
  Let $R$ be an NIP integral domain.
  \begin{enumerate}
  \item If $R$ is an $\Ff_p$-algebra, then $R$ is a henselian local
    ring \cite[Theorem~3.22]{fpcase}.
  \item If $R$ has finite dp-rank, then $R$ is a henselian local ring
    \cite[Theorem~1.3]{noetherian}.
  \end{enumerate}
\end{fact}
See \cite[Section~4.2]{NIPguide} for background on dp-rank.  For our
purposes, it suffices to think of dp-rank as a notion of dimension on
definable sets in NIP theories, similar to Morley rank in totally
transcendental theories.

The topologies induced by Henselian local domains happen to satisfy
the following property, isolated by Dittman, Walsberg, and Ye
\cite{hensquot2}:
\begin{definition} \label{pre-def}
  A topological field $(K,\tau)$ is \emph{generalized t-henselian} (or
  \emph{gt-henselian}) if it satisfies the implicit function theorem
  for polynomials: let $f_1,\ldots,f_m$ be polynomials in the
  variables $x_1,\ldots,x_n,y_1,\ldots,y_m$, and let $V$ be the
  zero-set $\{(\bx,\by) : f_1(\bx,\by) = \cdots = f_m(\bx,\by) = 0\}$.
  Let $(\bar{p},\bar{q}) \in K^{n+m}$ be a point at which the Jacobian
  matrix $\partial f_i / \partial y_j$ is invertible.  Then there are
  open neighborhods $U_1 \ni \bar{p}$ and $U_2 \ni \bar{q}$ such that
  $V \cap (U_1 \times U_2)$ is the graph of a continuous function $U_1
  \to U_2$.
\end{definition}
(Definition~\ref{pre-def} agrees with
\cite[Definition~8.1]{hensquot2} by Proposition~\ref{understand}
below.)  We will say more about gt-henselianity in Section~\ref{gth-sec}.
%% This prompted the following ``generalized henselianity conjecture,''
%% which generalizes the conjecture that NIP valuation rings are
%% henselian:
%% \begin{conjecture}[{\cite[Conjecture~1.2]{noetherian}}] \label{ghc}
%%   If $R$ is an NIP integral domain, then $R$ is a henselian local ring.
%% \end{conjecture}
Applying the machinery of Theorem~\ref{main-thm}, we get the following
corollary:
\begin{theorem}[{$\subseteq$ Proposition~\ref{gt-hens-prop}}]\label{gt-hens-thm}
  Let $K$ be an NIP expansion of a field and let $\tau$ be a definable
  field topology on $K$.  Suppose one of the following holds:
  \begin{enumerate}
  \item $\characteristic(K) = p$.
  \item $K$ has finite dp-rank.
%%   \item Conjecture~\ref{ghc} holds.
  \end{enumerate}
  Then $\tau$ is gt-henselian.
%%   Then the implicit function theorem for polynomials holds: let
%%   $f_1,\ldots,f_m$ be polynomials in the variables
%%   $x_1,\ldots,x_n,y_1,\ldots,y_m$, and let $V$ be the zero-set
%%   $\{(\bx,\by) : f_1(\bx,by) = \cdots = f_m(\bx,by) = 0\}$.  Let
%%   $(\bar{p},\bar{q}) \in K^{n+m}$ be a point at which the Jacobian
%%   matrix $\partial f_i / \partial y_j$ is invertible.  Then there are
%%   open neighborhods $U_1 \ni \bar{p}$ and $U_2 \ni \bar{q}$ such that
%%   $V \cap (U_1 \times U_2)$ is the graph of a continuous function $U_1 \to
%%   U_2$.
\end{theorem}
%% The implicit function theorem just described is equivalent to
%% \emph{generalized t-henselianity} (gt-henselianity) in the sense of
%% Dittman, Walsberg, and Ye \cite{hensquot2}.  We will say more about
%% this in Section~\ref{gt-hens-sec}.

%% With some additional work, we get two new reformulations of the
%% generalized henselianity conjecture:
%% \begin{theorem} \label{equivalence}
%%   The following are equivalent:
%%   \begin{enumerate}
%%   \item Every NIP integral domain $R$ is a henselian local ring.
%%   \item If $(K,+,\cdot,\ldots)$ is an NIP expansion of a field, and
%%     $\tau$ is a definable field topology on $K$, then $\tau$ is
%%     gt-henselian.
%%   \item Every NIP local ring $R$ is henselian.
%%   \end{enumerate}
%% \end{theorem}
%% Condition (2) is a topological reformulation of the generalized
%% henselianity conjecture.  Condition (3) is again purely algebraic, but
%% interestingly, the proof of the equivalence goes through Condition
%% (2).  We prove Theorem~\ref{equivalence} in
%% Section~\ref{sec-reformulate}.

Pop~\cite{Pop-little} says that a field $K$ is \emph{large} if for any
smooth algebraic curve $C$ over $K$,
\begin{equation*}
  |C(K)| > 0 \implies |C(K)| = \infty.
\end{equation*}
Conjecturally, any stable or NIP field should be large or finite.
Specifically, the stable fields conjecture implies that stable
fields are large or finite, while the Shelah conjecture on NIP fields
\cite[Conjecture~1.1]{NIP-char} implies the same for NIP fields.

Any field admitting a gt-henselian topology is large, essentially by
\cite{Pophenselian}; see Fact~\ref{MCF}(3).  Thus
Theorem~\ref{gt-hens-thm} has the following corollary:
\begin{corollary}[{$\subseteq$ Proposition~\ref{gt-hens-prop}}] \label{large-cor}
  Let $K$ be an NIP expansion of a field and let $\tau$ be a definable
  field topology on $K$.  Suppose one of the following holds:
  \begin{enumerate}
  \item $\characteristic(K) = p$.
  \item $K$ has finite dp-rank.
%%   \item Conjecture~\ref{ghc} holds.
  \end{enumerate}
  Then $K$ is large.
\end{corollary}
%%   \emph{If} we could prove Conjecture~\ref{ghc}, and \emph{if} we
%%   could construct enough definable field topologies (see
%%   Remark~\ref{conj-tops-exist}), this would imply that any NIP field
%%   is large or finite.  Large stable fields are classified
%%   \cite{firstpaper}, and \emph{if} we could generalize this
%%   classification to NIP fields, then we might hope to classify NIP
%%   fields.

%%   On the other hand, the rational function field $\Cc(t)$ has some
%%   suspicious properties suggesting that it is stable.  This would be a
%%   non-large infinite stable field, contradicting the stable fields
%%   conjecture, and \emph{a fortiori} the Conjecture on NIP fields.  In the
%%   event that these conjectures fail, Corollary~\ref{large-cor} has an
%%   interesting consequence: if $K$ is a non-large NIP field, then $K$
%%   cannot admit a definable topology (or $\characteristic(K) = 0$ and
%%   Conjecture~\ref{ghc} fails).
%% \end{remark}

\subsection{Finite breadth} \label{subsec3}
Recall that a topological field $(K,\tau)$ is \emph{V-topological}
\cite[Section~3]{PZ} if for any subset $B \subseteq K^\times$, $B$ is
bounded if and only if $0$ is not in the closure of $B^{-1}$.  A
theorem of Kowalsky, D\"urbaum, and Fleischer \cite{kow-dur,fleischer}
shows that the V-topologies on $K$ are precisely the field topologies
arising from (Krull) valuations and absolute values (i.e., norms) on
$K$.

A structure $M$ is \emph{dp-minimal} if it has dp-rank 1
\cite{dpExamples}.  Halevi and d'Elb\'ee show that any dp-minimal
integral domain with an infinite residue field is a valuation ring
\cite[Theorem~4.1]{halevi-delbee}.  By combining this with
Theorem~\ref{main-thm} and a little more work, one can prove:
\begin{fact}[{$\subseteq$ Theorem~\ref{wn-thm}}] \label{1-v}
  If $(K,+,\cdot,\ldots)$ is a dp-minimal expansion of a field, and
  $\tau$ is a definable field topology on $K$, then $\tau$ is a
  V-topology.
\end{fact}
This was already known \cite[Lemma~8.5]{noetherian}, but the proof we
give here is simpler, avoiding the use of the canonical topology from
\cite{dpm1}.

There is a natural generalization of Fact~\ref{1-v} from dp-rank 1 to
dp-rank $n$.  Following \cite[Definitions~3.2, 3.3, 3.9]{prdf5}, say
that a field topology $\tau$ on $K$ has \emph{breadth $\le n$} if
there is a bounded neighborhood $U \ni 0$ such that for any
$a_1,\ldots,a_{n+1} \in K$, there is some $i$ such that
\begin{equation*}
  a_i \in a_1 \cdot U_1 + \cdots + \widehat{a_i \cdot U_i} + \cdots +
  a_n \cdot U_n,
\end{equation*}
where the hat indicates omission.  This definition is quite opaque,
but we review the motivation and basic properties in
Section~\ref{sec-fin-b}.  For now, we mention that $\tau$ has breadth
1 iff $\tau$ is a V-topology, meaning that $\tau$ is induced by a
valuation or absolute value.

Combining Theorem~\ref{main-thm} with results in \cite{noetherian}, we
get the following corollary:
\begin{theorem}[{= Theorem~\ref{wn-thm}}] \label{wn-case}
  Let $K$ be an expansion of a field, with finite dp-rank $n <
  \infty$.  Let $\tau$ be a definable field topology on $K$.  Then
  $\tau$ has breadth $\le n$.
\end{theorem}
When $n = 1$, this recovers Fact~\ref{1-v}, the fact that definable
field topologies on dp-minimal fields are V-topological.  When $n >
1$, there are examples of definable topologies $\tau$ that are not
V-topological (Example~\ref{non-V-example}), at least when
$\characteristic(K) = 0$.

The situation in positive characteristic is very different.  In his
master's thesis \cite{yy}, Yang Yang uses Theorem~\ref{wn-case} to
prove the following surprising result:
\begin{fact}[Yang]
  If $\characteristic(K) = p > 0$ and $K$ has finite dp-rank, then any
  definable field topology $\tau$ on $K$ is a V-topology.
\end{fact}

\subsection{Applications to the generalized henselianity conjecture} \label{subsec4}
Recall the following two essential conjectures on NIP fields:
\begin{conjecture}[Shelah conjecture] \label{shelconj}
  If $K$ is an NIP field, then $K$ is finite or separably closed or
  real closed or $K$ admits a definable henselian valuation.
\end{conjecture}
\begin{conjecture}[Henselianity conjecture] \label{hensconj}
  If $(K,v)$ is an NIP valued field, then $v$ is henselian.
\end{conjecture}
The Shelah conjecture (also known as the ``Conjecture on NIP fields'')
implies a complete classification of NIP fields
\cite[Theorem~7.1]{NIP-char}, but is only known in the case
of finite dp-rank \cite[Theorem~1.2]{prdf6}.  The henselianity
conjecture is known in the case of finite dp-rank
\cite[Theorem~1.1]{prdf6} as well as the case of positive
characteristic \cite[Theorem~2.8]{prdf1a}.  Halevi, Hasson, and Jahnke
have shown that the Shelah conjecture implies the henselianity
conjecture \cite[Proposition~6.3]{hhj-v-top}, though in practice the
expectation is that the henselianity conjecture would be an ingredient
in the proof of the Shelah conjecture (see the discussion following
\cite[Conjecture~2.2]{halevi-hasson-jahnke}).

In \cite[Conjecture~1.2]{noetherian}, I proposed the following
generalization of the henselianity conjecture:
\begin{conjecture}[Generalized henselianity conjecture] \label{ghcv2}
  The following (equivalent) statements hold:
  \begin{enumerate}
  \item \label{old1} If $R$ is an NIP integral domain, then $R$ is a henselian
    local ring.
  \item \label{old2} If $R$ is an NIP integral domain, then $R$ is a local ring.
  \item \label{old3} If $R$ is an NIP ring, then $R$ is a direct product of
    finitely many henselian local rings.
  \end{enumerate}
\end{conjecture}
The conditions are equivalent by \cite[Proposition~3.6]{noetherian}.
This conjecture---in all three forms---holds when $R$ has finite
dp-rank \cite[Theorem~1.3]{noetherian} or when $R$ is an
$\Ff_p$-algebra \cite[Theorem~3.22]{fpcase}.

Using Theorem~\ref{main-thm} and some additional work, we give a few
more reformulations of the generalized henselianity conjecture:
\begin{theorem}[{= Theorems \ref{topreform} and \ref{local-form}}] \label{shiliu}
  Conditions (\ref{old1})--(\ref{old3}) of Conjecture~\ref{ghcv2} are equivalent to the
  following statements:
  \begin{enumerate}
    \setcounter{enumi}{3}
  \item \label{old4} If $(K,+,\cdot,\ldots)$ is an NIP expansion of a
    field, and $\tau$ is a definable field topology on $K$, then
    $\tau$ is gt-henselian---it satisfies the implicit and inverse
    function theorems for polynomials.
  \item \label{old4.5} If $(K,+,\cdot,\ldots)$ is an NIP expansion of
    a field, and $\tau_1, \tau_2$ are two definable field topologies on $K$,
    then $\tau_1$ and $\tau_2$ are \emph{not} independent.
  \item \label{old5} If $R$ is an NIP local ring, then $R$ is henselian.
  \item \label{old6} If $R$ is an NIP local integral domain, then $R$ is henselian.
  \end{enumerate}
\end{theorem}
Conditions (\ref{old4}) and (\ref{old4.5}) give topological
reformulations of the generalized henselianity conjecture, and
Condition (\ref{old4}) in particular seems conceptually very nice.
Condition (\ref{old4.5}) is notable for suggesting an approach to
prove the generalized henselianity conjecture \emph{assuming} the
Shelah conjecture, which is discussed in Section~\ref{thoughts} below.

Conditions (\ref{old5}) and (\ref{old6}) are purely algebraic, but the
proof of their equivalence to (\ref{old1})--(\ref{old3}) uses
topologies.  If we phrase the (original) henselianity conjecture as
``NIP valuation rings are henselian'', then Condition (\ref{old5}) is
arguably a better generalization of the henselianity conjecture than
Condition (\ref{old1}), so it is nice to know that the two statements
are equivalent.

%% Note that the implication (\ref{old1})$\implies$(\ref{old4})
%% is just like Theorem~\ref{gt-hens-thm}.

\section{From rings to field topologies} \label{frtft}
In this section, we review some background material, almost all of
which is from \cite{PZ}.

Let $K$ be a field.  A \emph{ring topology} on $K$ is a non-discrete
non-trivial (i.e., non-indiscrete) topology on $K$ such that the ring
operations $+,-,\cdot$ are continuous.  Any ring topology is
Hausdorff.  If $K$ admits a ring topology, then $K$ must be infinite.
A \emph{field topology} is a ring topology such that division is
continuous.  A ring topology is determined by the filter of
neighborhoods of 0, which determines every other neighborhood filter
by translation.

Fix a field $K$ with a ring topology $\tau$.  A set $B \subseteq K$ is
\emph{bounded} if for any neighborhood $U \ni 0$, there is some $c \in
K^\times$ with $cB \subseteq U$.
\begin{fact}[{\cite[Lemma~2.1]{PZ}}]
  ~ \label{pz-bdd}
  \begin{enumerate}
  \item $K$ is unbounded.  
  \item Finite sets are bounded.
  \item Subsets of bounded sets are bounded.
  \item If $B_1, B_2$ are bounded, then the sum set $B_1 + B_2$, the
    product set $B_1 \cdot B_2$, and the union $B_1 \cup B_2$ are
    bounded.
  \item If $U \ni 0$ is a bounded neighborhood of 0, then $\{cU : c
    \in K^\times\}$ is a neighborhood basis of 0.
  \item If $U \ni 0$ is a bounded neighborhood of 0, then $\{cU : c
    \in K^\times\}$ is a basis for the ideal of bounded sets.  That
    is, $B \subseteq K$ is bounded if and only if $B \subseteq cU$ for
    some $c \in K^\times$.
  \end{enumerate}
\end{fact}
A ring topology $\tau$ is \emph{locally bounded} if some neighborhood
of 0 is bounded, in which case statements (5) and (6) take effect.
\begin{fact}[{\cite[Example~1.2]{PZ}}] \label{r-adic}
  Let $R$ be an integral domain with fraction field $\Frac(R) = K \ne
  R$.  Then there is a locally bounded ring topology $\tau_R$ on $K$
  such that
  \begin{enumerate}
  \item The set of non-zero ideals of $R$ is a neighborhood basis of
    0.
  \item The set $\{cR : c \in K^\times\}$ is a neighborhood basis of
    0.
  \item $R$ is a bounded neighborhood of 0.
  \end{enumerate}
\end{fact}
Fact~\ref{r-adic} is
mentioned in \cite[Example~1.2]{PZ} and the proof of
\cite[Theorem~2.2(a)]{PZ}, and is easy to verify directly.  For
example, the two neighborhood bases are equivalent for the following
reasons:
\begin{itemize}
\item If $I$ is a non-zero ideal in $R$, then $I \supseteq cR$ for any
  $c \in I \setminus \{0\}$.
\item If $a,b \in R \setminus \{0\}$, then $ab^{-1}R$ contains the
  non-zero ideal $aR$.
\end{itemize}
We call $\tau_R$ the \emph{$R$-adic topology}, following \cite{mana}.
\begin{remark} \label{aff-basis}
  The set $\{aR + b : a \in K^\times, ~ b \in K\}$ is a basis of open
  sets for the $R$-adic topology.  Consequently, the $R$-adic topology
  is definable in the pair $(K,R)$.
\end{remark}
Recall that the \emph{Jacobson radical} $J(R)$ of a commutative unital
ring $R$ is the intersection of all maximal ideals of $R$, which can
also be characterized as the largest ideal $J$ such that $1 + J
\subseteq R^\times$ (see \cite[Exercise~7.3]{eisenbud}).  We will need
the following fact:
\begin{fact} \label{j-fact}
  The $R$-adic topology $\tau_R$ is a field topology if and only if
  the Jacobson radical $J(R)$ is non-trivial.
\end{fact}
This is implicit in the proof of \cite[Theorem~2.2(b)]{PZ}, but we
include the proof for completeness.
\begin{proof}
  First suppose $J(R) \ne 0$.  We claim that division is continuous.
  It suffices to show that the map $f(x) = 1/x$ is continuous.  Since
  $\tau_R$ is a ring topology, the map $x \mapsto ax$ is continuous
  for any $a \in K$.  Using these maps, we reduce to showing that $f$
  is continuous at $x = 1$.  Let $U$ be a neighborhood of 1.  Then $U
  \supseteq 1 + I$ for some non-zero ideal $I$ in $R$.  The
  intersection $I \cap J(R)$ is also non-zero (see
  Remark~\ref{intersection} below), so replacing $I$ with $I \cap
  J(R)$ we may assume $I \subseteq J(R)$.  Then $1 + I \subseteq
  R^{\times}$.  In particular, if $x \in 1+I$ then $x$ is a unit, so
  \begin{equation*}
    1/x - 1 = \frac{1-x}{x} = x^{-1}(1-x) \in x^{-1}I = I,
  \end{equation*}
  and $1/x \in 1 + I \subseteq U$.  Therefore $f(x) = 1/x$ maps the
  neighborhood $1+I$ into $U$, proving continuity.

  Conversely, suppose that division is continuous.  The ring $R$ is a
  neighborhood of 1, so there must be some neighborhood $U$ of 1 such
  that $x \in U \implies x^{-1} \in R$.  Shrinking $U$, we may assume
  $U = 1 + I$ for some non-zero ideal $I$.  If $x \in I$, then $1+x
  \in R$ and $(1+x)^{-1} \in R$ by choice of $U$.  Therefore $1+I
  \subseteq R^\times$, and $I$ is contained in the Jacobson radical
  $J(R)$, which must be non-zero.
\end{proof}
\begin{remark} \label{intersection}
  In an integral domain, the intersection of two non-zero ideals $I$
  and $J$ is itself a non-zero ideal, because if $a \in I \setminus
  \{0\}$ and $b \in J \setminus \{0\}$, then $ab \in I \cap J$.
\end{remark}
In NIP integral domains, the Jacobson radical is always non-trivial:
\begin{lemma} \label{simon-sorta}
  Let $R$ be an NIP integral domain that is not a field.  Then $J(R)
  \ne 0$.
\end{lemma}
\begin{proof}
  By a theorem of Simon \cite[Proposition~2.1]{halevi-delbee}, $R$ has
  only finitely many maximal ideals $\mm_1, \ldots, \mm_n$.  No
  $\mm_i$ can vanish, or else $R$ would be a field.  Then the Jacobson
  radical is an intersection of finitely many non-zero ideals, and is
  therefore non-zero by Remark~\ref{intersection}.
\end{proof}
Combining Lemma~\ref{simon-sorta}, Fact~\ref{j-fact}, and
Remark~\ref{aff-basis}, we conclude:
\begin{corollary} \label{new-cor}
  If $R$ is an NIP integral domain and $K = \Frac(R) \ne R$, then the
  $R$-adic topology $\tau_R$ is a field topology, and $\tau_R$ is
  definable in the pair $(K,R)$.
\end{corollary}
%% Consequently, the $R$-adic topology is a field topology when $R$ is
%% NIP.  With Remark~\ref{aff-basis} for definability, this completes the
%% proof of Fact~\ref{main-fact} in the introduction.

\section{Local boundedness} \label{sec:locbound}
In this section, we prove the core technical result of the paper, the
fact that definable ring topologies on NIP fields are locally bounded (Theorem~\ref{loc-bdd-thm2}).
We first need some combinatorial lemmas about vector spaces.  Fix an
\emph{infinite} field $K_0$.  The following is well-known:
\begin{fact} \label{cover-fail}
  If $V$ is a $K_0$-vector space, then $V$ is not a finite union of
  proper $K_0$-linear subspaces.
\end{fact}
More generally, no group is a finite union of infinite index subgroups or their cosets \cite[Lemma~6.25]{P-book}.
\begin{lemma} \label{cover-rule}
  Let $V$ and $V'$ be two subspaces of a $K_0$-vector space.  Let
  $W_1,\ldots,W_n$ be proper subspaces of $V$.  Then
    \begin{equation*}
      V' \supseteq V \setminus (W_1 \cup \cdots \cup W_n) \iff V'
      \supseteq V.
    \end{equation*}
\end{lemma}
The intuition is that subspaces are ``closed'', and $V$ is the
``closure'' of $V \setminus (W_1 \cup \cdots \cup W_n)$.
\begin{proof}
  The right-to-left direction is trivial.  Suppose the left-to-right
  direction fails.  Then $X := V' \cap V$ is a proper subspace of $V$,
  and $V = W_1 \cup \cdots \cup W_n \cup X$, contradicting
  Fact~\ref{cover-fail}.
\end{proof}
Fix some set $U$.  A family of subsets $V_1,\ldots,V_n \subseteq U$ is
\emph{independent} if for any $S \subseteq \{1,\ldots,n\}$, the set
$A_S$ is non-empty, where
\begin{align*}
  A_S & = \bigcap_{i \in S} V_i \setminus \bigcup_{i \notin S} V_i \\
  &= \{x \in U : x \in A_i \iff i \in S \text{ for all } i \le n\}.
\end{align*}
Equivalently, the family $V_1,\ldots,V_n$ is
independent if there exist $a_S \in U$ for each $S \subseteq
\{1,\ldots,n\}$ such that
\begin{equation*}
  a_S \in V_i \iff i \in S \text{ for all } i, S.
\end{equation*}
\begin{lemma} \label{extend-1}
  Let $V_1,\ldots,V_n$ be subspaces of some $K_0$-vector space.
  Suppose $V_1,\ldots,V_n$ are independent, witnessed by elements $a_S$ for $S
  \subseteq \{1,\ldots,n\}$.  Suppose $W$ is some other subspace, such
  that
  \begin{gather*}
    a_S \in W \text{ for each } S \\
    W \not \supseteq V_1 \cap \cdots \cap V_n.
  \end{gather*}
  Then the family $V_1,\ldots,V_n,W$ is independent.
\end{lemma}
\begin{proof}
  For $S \subseteq \{1,\ldots,n\}$, let $A_S$ be the corresponding
  boolean combination of $V_1,\ldots,V_n$:
  \begin{equation*}
    A_S = \bigcap_{i \in S} V_i \setminus \bigcup_{i \notin S} V_i.
  \end{equation*}
  In particular, $a_S \in A_S$.  Also let
  \begin{equation*}
    \overline{A_S} = \bigcap_{i \in S} V_i.
  \end{equation*}
  Then $\overline{A_S}$ is a linear subspace of the ambient vector
  space.  Note that
  \begin{equation*}
    A_S = \overline{A_S} \setminus \bigcup_{i \notin S} V_i = \overline{A_S} \setminus \bigcup_{i \notin S} (\overline{A_S} \cap V_i).
  \end{equation*}
%%   \begin{equation*}
%%     A_S = \overline{A_S} \setminus \bigcup_{S' \supseteq S}
%%     \overline{A_{S'}}.
%%   \end{equation*}
  By Lemma~\ref{cover-rule},
  \begin{equation*}
    W \supseteq A_S \iff W \supseteq \overline{A_S}.
  \end{equation*}
  But $\overline{A_S}$ contains the intersection $V_1 \cap \cdots \cap
  V_n$, which $W$ does not contain.  Therefore $W \not \supseteq A_S$
  for each $S$, meaning that $A_S \setminus W \ne \varnothing$ for
  each $S$.  On the other hand, the element $a_S$ shows that $A_S \cap
  W \ne \varnothing$ for each $S$.  It follows that
  $\{V_1,\ldots,V_n,W\}$ is independent.
\end{proof}
\begin{theorem} \label{loc-bdd-thm2}
  Let $\tau$ be a definable ring topology on an NIP field
  $(K,+,\cdot,\ldots)$, possibly with extra structure.  Then $\tau$ is
  locally bounded.
\end{theorem}
\begin{proof}
  If $L$ is an elementary extension of $K$, then the formulas defining
  $\tau$ on $K$ define some ring topology $\tau_L$ on $L$, and
  $\tau_L$ is locally bounded if and only if $\tau$ is locally bounded
  \cite[\S 2, p.\@ 322]{PZ}.  Consequently, we may assume that $K$ is
  highly saturated.

  Fix a definable basis $\mathcal{B}$ of opens.  Replacing
  $\mathcal{B}$ with $\{aU + b : a \in K^\times, ~ U \in \mathcal{B},
  ~ b \in K\}$, we may assume that the image of a basic open set under
  an affine transformation is still basic.  A \emph{basic
  neighborhood} will be a basic open set containing 0.  The basic
  neighborhoods constitute a definable basis of neighborhoods of 0.  If
  $U$ is a basic neighborhood, and $c \in K^\times$, then $cU$ is also
  a basic neighborhood.

  Because $K$ admits a ring topology, it is infinite.  Fix a countably
  infinite subfield $K_0$.  A \emph{special neighborhood} is a set $G
  \subseteq K$ with the following properties:
  \begin{enumerate}
  \item There is a descending chain of basic neighborhoods $U_0
    \supseteq U_1 \supseteq \cdots$ (of length $\omega$) such that $G
    = \bigcap_{i = 0}^\infty U_i$.
  \item $G$ is a $K_0$-linear subspace of $K$.
  \end{enumerate}
  \begin{claim} \label{claim3.5}
    Any special neighborhood $G$ is a neighborhood of 0.
  \end{claim}
  \begin{claimproof}
    Writing $G$ as $\bigcap_{i = 0}^\infty U_i$ for some basic
    neighborhoods $U_0 \supseteq U_1 \supseteq \cdots$, we can use
    saturation to find a basic neighborhood $U_\omega \subseteq \cdots
    \subseteq U_1 \subseteq U_0$.  Then $U_\omega \subseteq G$, so $G$
    is a neighborhood of 0.
  \end{claimproof}
  \begin{claim} \label{enough-special}
    Special neighborhoods form a neighborhood basis of 0.  That is,
    any neighborhood $U \ni 0$ contains a special neighborhood $G$.
  \end{claim}
  \begin{claimproof}
    Take a basic neighborhood $U_0 \subseteq U$.  Let
    $\{a_1,a_2,\ldots\}$ enumerate $K_0$.  Recursively build a chain
    of basic neighborhoods $U_1, U_2, U_3, \ldots$, so that
    \begin{gather*}
      U_i \subseteq U_{i-1} \\
      U_i - U_i \subseteq U_{i-1} \\
      a_jU_i \subseteq U_{i-1} \text{ for } j = 1, \ldots, i
    \end{gather*}
%%     \begin{equation*}
%%       U_{i-1} \supseteq U_i \cup (U_i - U_i) \cup
%%       \bigcup_{j = 1}^i (a_j \cdot U_i)
%%     \end{equation*}
    for each $i$.  At each step, it is possible to choose $U_i$ by
    continuity of the ring operations---any sufficiently small $U_i$
    will work.  Take $G = \bigcap_{i = 1}^\infty U_i$.  Then $G$ is an
    intersection of a descending chain of basic neighborhoods.
    Clearly $0 \in G$.  If $x,y \in G$ and $i < \omega$, then
    \begin{equation*}
      x-y \in G-G \subseteq U_{i+1} - U_{i+1} \subseteq U_i \text{ for any } i,
    \end{equation*}
    and so $x-y \in G$.  This shows that $G-G \subseteq G$, and $G$ is
    closed under subtraction.  Finally, for any $j < \omega$ and $x \in
    G$, we have
    \begin{equation*}
      a_j \cdot x \in a_j \cdot U_{i+1} \subseteq U_i \text{ for any } i \ge j
    \end{equation*}
    implying that $a_j x \in G$ and $G$ is closed under multiplication
    by $K_0$.  Thus $G$ is a special neighborhood and $G \subseteq U_0
    \subseteq U$.
  \end{claimproof}
  \begin{claim} \label{phew}
    If $G$ is a special neighborhood and $c \in K^\times$, then $cG$
    is a special neighborhood.
  \end{claim}
  \begin{claimproof}
    This follows easily from the analogous fact for basic
    neighborhoods, which we arranged.
  \end{claimproof}
  Now, suppose for the sake of contradiction that $\tau$ is not
  locally bounded.  Then no neighborhood $U \ni 0$ is bounded.  By
  definition of ``bounded'', this means the following:
  \begin{quote}
    If $U$ is a neighborhood of 0, then there is a
    neighborhood $V \ni 0$ such that for any $c \in K^\times$, $cU \not
    \subseteq V$.
  \end{quote}
  Shrinking $V$, we may assume that $V$ is a special neighborhood
  (Claim~\ref{enough-special}).
  \begin{claim} \label{claim3.8}
    For any $n$, there is an independent sequence of special
    neighborhoods $G_1,\ldots,G_n$.
  \end{claim}
  \begin{claimproof}
    Build the sequence by induction on $n$.  When $n = 0$ there is
    nothing to build.  Suppose $n > 0$, and let $G_1,\ldots,G_{n-1}$
    be a given independent sequence of special neighborhoods.  Choose
    elements $a_S$ for $S \subseteq \{1,\ldots,n-1\}$ witnessing
    independence, meaning that
    \begin{equation*}
      a_S \in G_i \iff i \in S.
    \end{equation*}
    Replacing $a_{\{1,\ldots,n-1\}}$ with a non-zero element of the
    neighborhood $G_1 \cap \cdots \cap G_{n-1}$, we may assume that
    every $a_S$ is non-zero.

    By Lemma~\ref{extend-1}, it suffices to produce a special
    neighborhood $G_n$ such that
    \begin{gather*}
      a_S \in G_n \text{ for every } S \\
      G_n \not \supseteq G_1 \cap \cdots \cap G_{n-1}.
    \end{gather*}
    The set $G_1 \cap \cdots \cap G_{n-1}$ is a neighborhood of 0.  By
    failure of local boundedness, it is not bounded.  Therefore, there
    is a special neighborhood $V$ such that
    \begin{equation*}
      c(G_1 \cap \cdots \cap G_{n-1}) \not \subseteq V \text{ for any
      } c \in K^\times.
    \end{equation*}
    Take non-zero $c$ in the neighborhood
    \begin{equation*}
      \bigcap_{S \subseteq \{1,\ldots,n-1\}} a_S^{-1} \cdot V.
    \end{equation*}
    Then $c \in a_S^{-1} V$ for each $S$, implying $a_S \in c^{-1}V$.
    Take $G_n = c^{-1}V$; this is special by Claim~\ref{phew}.  Then $a_S \in G_n$ for each $S$, and
    \begin{equation*}
      G_1 \cap \cdots \cap G_{n-1} \not \subseteq c^{-1}V = G_n
    \end{equation*}
    by choice of $V$.
  \end{claimproof}
  \begin{claim}\label{at-last}
    For any $n$, there is an independent sequence of basic
    neighborhoods $U_1, \ldots, U_n$.
  \end{claim}
  \begin{claimproof}
    Take an independent sequence of special neighborhoods
    $G_1,\ldots,G_n$, with independence witnessed by $a_S$ for $S
    \subseteq \{1,\ldots,n\}$, so that
    \begin{equation*}
      a_S \in G_i \iff i \in S.
    \end{equation*}
    Write $G_i$ as an intersection of a decreasing chain of basic
    neighborhoods $G_i = \bigcap_{j = 0}^\infty U_{i,j}$ with $U_{i,0}
    \supseteq U_{i,1} \supseteq \cdots$.

    For fixed $i$ and $S$, note that
    \begin{equation*}
      a_S \in U_{i,j} \iff a_S \in G_i \iff i \in S, \text{ for all
        sufficiently large $j$.}
    \end{equation*}
    As there are only finitely many pairs $(i,S)$, we can fix some $j$ so
    large that it works for every $(i,S)$.  Then
    \begin{equation*}
      a_S \in U_{i,j} \iff i \in S \text{ for any $i$, $S$},
    \end{equation*}
    and the family $U_{1,j},\ldots,U_{n,j}$ is independent.
  \end{claimproof}
  Finally, Claim~\ref{at-last} contradicts NIP, as the basic
  neighborhoods are uniformly definable.
\end{proof}

\section{The main theorem} \label{secmt}
\begin{theorem} \label{main-thm2}
  Let $K$ be a sufficiently saturated NIP field, possibly with extra
  structure.  Let $\tau$ be a definable ring topology on $K$.  Then
  there is an externally definable proper subring $R \subsetneq K$
  with $\Frac(R) = K$, such that $\tau$ is the $R$-adic topology
  $\tau_R$.

  We can also arrange for $R$ to be a $K_0$-algebra for some small
  elementary substructure $K_0 \preceq K$.
\end{theorem}
% TODO: comment that this technique is well-known...
\begin{proof}
  Fix a definable basis of open sets.  By Theorem~\ref{loc-bdd-thm2},
  the topology is locally bounded, and so there is some bounded
  neighborhood $U \ni 0$.  Shrinking $U$, we may assume that $U$ is
  basic.  Then $U$ is definable.  By Fact~\ref{pz-bdd}(6), the family
  of sets $\{cU : U \in K^\times\}$ is a basis for the ideal of
  bounded sets.

  Fix a small elementary substructure $K_0$ defining $U$ and the
  topology $\tau$.  Note the following consequences of
  Fact~\ref{pz-bdd}:
  \begin{itemize}
  \item If $B \subseteq K$ is $K_0$-definable, then $B$ is bounded if
    and only if $B \subseteq cU$ for some $c \in K_0^\times$.  This is
    essentially Fact~\ref{pz-bdd}(6); we can arrange for $c$ to be in
    $K_0$ because $K_0 \preceq K$.
  \item If $B_1, B_2 \subseteq K$ are $K_0$-definable and bounded,
    then $B_1 + B_2$, $B_1 \cdot B_2$, and $B_1 \cup B_2$ are
    $K_0$-definable and bounded (Fact~\ref{pz-bdd}(4)).
  \item If $c \in K_0$, then $\{c\}$ is $K_0$-definable and bounded
    (Fact~\ref{pz-bdd}(2)).
  \end{itemize}
  Let $R$ be the union of all $K_0$-definable bounded sets.  By the
  above, $R$ is a subring of $K$ containing $K_0$, and
  \begin{equation*}
    R = \bigcup \{cU : c \in K_0^\times\}.
  \end{equation*}
  This union is directed: given $c_1, c_2 \in K_0^\times$, the set
  $c_1U \cup c_2U$ is bounded and $K_0$-definable, so it is contained
  in $c_3U$ for some $c_3 \in K_0^\times$.  A directed union of
  uniformly definable sets is always externally definable \cite[Remark~2.9]{fpcase}.  In
  summary, $R$ is an externally definable $K_0$-subalgebra of $K$.

  The ring $R$ is also $\vee$-definable, a small union of definable
  sets.  If $R = K$, then saturation would imply that $cU = K$ for
  some $c \in K_0^\times$.  But then $U = K$, and $K$ is bounded,
  contradicting Fact~\ref{pz-bdd}(1).  Therefore $R \subsetneq K$.

  On the other hand, $\Frac(R) = K$.  To see this, first note that $U
  = 1 \cdot U \subseteq R$.  Next, fix any $a \in K^\times$.  The set
  $U \cap aU$ is a neighborhood of 0, so it contains a non-zero
  element $\delta$ by non-discreteness.  Then $\delta$ and $\delta/a$
  are both in $U \subseteq R$, and so their quotient $a$ is in
  $\Frac(R)$.

  Because $R$ is a proper subring of $K$ and $\Frac(R) = K$, there is
  a well-defined $R$-adic topology $\tau_R$, characterized by the fact
  that $\{cR : c \in K^\times\}$ is a neighborhood basis of 0.  It
  remains to show that $\tau = \tau_R$.

  Recall that $R$ is the union of all $K_0$-definable bounded sets.
  Because every finite union of $K_0$-definable bounded sets is
  contained in a set of the form $cU$ with $c \in K_0^\times$,
  saturation gives some $c \in K^\times$ such that $R \subseteq cU$.
  Then $U \subseteq R \subseteq cU$, which implies $\tau = \tau_R$ since
  \begin{gather*}
    \{aU : a \in K^\times\} \text{ is a neighborhood basis for } \tau \\
    \{aR : a \in K^\times\} \text{ is a neighborhood basis for } \tau_R.  \qedhere
  \end{gather*}
\end{proof}

\begin{corollary} \label{cor:r-to-f}
  Let $\tau$ be a definable ring topology on an NIP field
  $(K,+,\cdot,\ldots)$, possibly with extra structure.  Then $\tau$ is
  a field topology (division is continuous).
\end{corollary}
\begin{proof}
  Passing to an elementary extension, we may assume that $K$ is highly
  saturated.  Then Theorem~\ref{main-thm2} gives an externally
  definable proper subring $R \subsetneq K$ with $K = \Frac(R)$ and
  $\tau = \tau_R$.  The ring $R$ is NIP, so $\tau_R$ is a field
  topology by Corollary~\ref{new-cor}.
\end{proof}

\section{Finite breadth} \label{sec-fin-b}
Fix $n \ge 1$.  We recall some definitions from
\cite[Sections~2--3]{prdf5}.  An integral domain $R$ is a
\emph{$W_n$-domain} if for any $a_0, a_1, \ldots, a_n \in R$, there is
some $i$ such that
\begin{equation*}
  a_i \in a_0 \cdot R + a_1 \cdot R + \cdots + \widehat{a_i \cdot R} + \cdots + a_n \cdot R.
\end{equation*}
where the $\widehat{\text{hat}}$ denotes omission.  The \emph{breadth} (or \emph{weight})
of a domain $R$, written $\br(R)$, is the minimum $n$ such that $R$ is
a $W_n$-domain, or $\infty$ if no such $n$ exists.

A \emph{$W_n$-set} on a
field $K$ is a set $X \subseteq K$ such that for any $a_0, a_1,
\ldots, a_n \in K$, there is some $i$ such that
\begin{equation*}
  a_i \in a_0 \cdot X + \cdots + \widehat{a_i \cdot X} + \cdots + a_n \cdot X.
\end{equation*}
We say that a ring topology $\tau$ on $K$ is a \emph{$W_n$-topology}
if $\tau$ is locally bounded, and there is a bounded $W_n$-set $B
\subseteq K$.  The \emph{breadth} (or \emph{weight}) of a ring
topology $\tau$ is the minimum $n$ such that $\tau$ is a
$W_n$-topology, or $\infty$ if no such $n$ exists.    Thus $\tau$ is a $W_n$-topology iff $\br(\tau) \le n$.  Recall the
following:
\begin{fact} \phantomsection \label{some-facts}
  \begin{enumerate}
\item The class of $W_n$-topologies is a local class
  \cite[Remark~3.4]{prdf5}, meaning it is defined by a local sentence
  in the sense of Prestel and Ziegler \cite[\S1]{PZ}.
\item If $K = \Frac(R) \supsetneq R$, and $R$ is a $W_n$-domain, then
  $R$ is a $W_n$-set on $K$ and the induced topology $\tau_R$ is therefore a
  $W_n$-topology \cite[Proposition~3.6]{prdf5}.
\item Up to local equivalence in the sense of Prestel and Ziegler \cite[p.~320]{PZ}, all $W_n$-topological fields arise via
  this construction \cite[Corollary~3.8(1)]{prdf5}.
\item Any $W_n$-topology is a field topology
  \cite[Corollary~3.8(2)]{prdf5}.
\item $W_1$-domains are the same thing as valuation rings.
  $W_1$-topologies are the same thing as V-topologies
  \cite[Proposition~2.6, Corollary~3.8(4)]{prdf5}.
  \end{enumerate}
\end{fact}
The following is \cite[Lemma~5.6]{noetherian}.
\begin{fact} \label{from-joh23}
	If $R$ is a ring and $R/\mm$ is infinite for every maximal ideal $\mm \lhd R$, then $\br(R) \le \dpr(R)$.
\end{fact}
\begin{lemma}\label{rank}
  Let $R$ be an integral domain with $\dpr(R) = n < \omega$.  Suppose
  that $R$ is a $K_0$-algebra for some infinite subfield $K_0
  \subseteq R$.  Then $\br(R) \le n$.
\end{lemma}
\begin{proof}
	If $\mm$ is a maximal ideal, then $R/\mm$ is a non-trivial $K_0$-vector space, and is therefore infinite.  Then Fact~\ref{from-joh23} applies.
\end{proof}
\begin{theorem} \label{wn-thm}
  Let $K$ be a field, possibly with extra structure.  Suppose $K$ has
  finite dp-rank $\dpr(K) = n < \omega$.  Let $\tau$ be a definable
  field topology on $K$.  Then $\tau$ is a $W_n$-topology.
\end{theorem}
\begin{proof}
  Let $L$ be a highly saturated elementary extension of $K$, and let
  $\tau_L$ be the canonical extension of $\tau$ to $K$.  Then
  $(L,\tau_L)$ and $(K,\tau)$ are locally equivalent in the sense of
  Prestel and Ziegler \cite{PZ}.  The definition of ``$W_n$-topology'' is a
  local sentence, so it suffices to show that $\tau_L$ is a
  $W_n$-topology.  Replacing $K$ with $L$, we may assume $K$ is highly
  saturated.

  Then Theorem~\ref{main-thm2} shows that $\tau$ is the topology
  $\tau_R$ induced by some externally definable subring $R$.
  Moreover, we can take $R$ to be a $K_0$-algebra for some small
  elementary substructure $K_0$.  The expansion $(K,R)$ has the same
  dp-rank as $K$ \cite[Fact~3.1]{dEHJ}, so its dp-rank is at most $n$.  The
  fields $K$ and $K_0$ must be infinite, since $K$ admits a field
  topology.  Then Lemma~\ref{rank} shows that $R$ is a $W_n$-domain.
  By Fact~\ref{some-facts}(2), $\tau = \tau_R$ is a $W_n$-topology.
\end{proof}

Most of the obvious examples of definable topologies on NIP fields are
V-topologies, i.e., $W_1$-topologies.  Below, we give an example from
\cite{prdf4} of a definable topology of breadth 2.
\begin{example} \label{non-V-example}
  Let $T_0$ be the theory of structures $(K,v,\partial)$, where $(K,v)
  \models \ACVF_{0,0}$, and $\partial : K \to K$ is a derivation.  By
  \cite{Michaux,GuzyPoint}, $T_0$ has a model companion $T$, and $T$ is NIP.  If
  $(K,v,\partial) \models T$, then there is a definable field topology
  $\tau_{\partial v}$ on $K$ characterized by the fact that sets of
  the form
  \begin{equation*}
    B_{a,b,\gamma} = \{x \in K : v(x-a) > \gamma \text{ and }
    v(\partial x - b) > \gamma\}
  \end{equation*}
  are a basis of open sets, and $\tau_{\partial v}$ is \emph{not} a
  V-topology \cite[Sections~8.4--8.5]{prdf4}.

  Let $R$ be the subring $R = \{x \in K : v(x) \ge 0 \text{ and }
  v(\partial x) \ge 0\}$.  Then $\tau_{\partial v}$ is the $R$-adic
  topology \cite[Definition~8.16, Proposition~8.21]{prdf4}, and $R$
  has breadth $\le 2$ by the proof of \cite[Lemma~5.8]{mana}.  It
  follows that $\tau_{\partial v}$ has breadth at most 2.  Since
  $\tau_{\partial v}$ is \emph{not} a V-topology, the breadth $\br(\tau_{\partial
    v})$ is exactly 2.

  This example has infinite dp-rank---it is not even strongly
  dependent, because of the ict-pattern
  \begin{equation*}
    \begin{pmatrix}
      0 < v(x) < 1 & 1 < v(x) < 2 & 2 < v(x) < 3 & \cdots \\
      0 < v(\partial x) < 1 & 1 < v(\partial x) < 2 & 2 < v(\partial x) < 3 & \cdots \\
      0 < v(\partial^2 x) < 1 & 1 < v(\partial^2 x) < 2 & 2 < v(\partial^2 x) < 3 & \cdots \\
      \vdots & \vdots & \vdots & \ddots 
    \end{pmatrix}
  \end{equation*}
  Nevertheless, in \cite[Theorem~10.1]{prdf4}, it is shown that the
  reduct $(K,R)$ has dp-rank 2.  This reduct defines $\tau_{\partial
    v} = \tau_R$.  This gives an example of a dp-rank 2 expansion of a
  field defining a topology of breadth $2$.  More generally, we expect
  there to be dp-rank $n$ structures $(K,+,\cdot,\ldots)$ with
  definable topologies of breadth $n$.
\end{example}

\section{Generalized t-henselianity} \label{gth-sec}
Prestel and Ziegler say that a topological field is
\emph{topologically henselian} or \emph{t-henselian} if it is locally
equivalent to a topological field $(K,\tau)$ where $\tau$ is induced
by a henselian valuation ring on $K$ \cite[Section~7]{PZ}.  Three
natural examples of t-henselian topological fields are (1) real closed
fields with the order topology, (2) $\Cc$ with the analytic topology,
and (3) henselian valued fields with the valuation topology
\cite[Corollary~7.3]{PZ}.

Dittmann, Walsberg, and Ye introduce the following generalization of
t-henselianity: \cite[Definition~8.1]{hensquot2}
\begin{definition}\label{original-def}
  A topological field $(K,\tau)$ is \emph{generalized t-henselian} (or
  \emph{gt-henselian}) if the following holds: for any neighborhood
  $U$ of $-1$ and any $n \ge 2$, there is a neighborhood $V$ of 0 such
  that if $c_0, c_1, \ldots, c_{n-2} \in V$, then the polynomial $X^n
  + X^{n-1} + c_{n-2}X^{n-2} + c_{n-1}X^{n-1} + \cdots + c_1X + c_0$
  has a root in $U$.
\end{definition}
The next proposition gives some alternate characterizations of
gt-henselianity that are more natural:
\begin{proposition} \label{understand}
  Let $(K,\tau)$ be a topological field.  The following are
  equivalent:
  \begin{enumerate}
  \item $\tau$ is gt-henselian.
  \item \label{etale} If $f : V \to W$ is an etale morphism of varieties over $K$,
    then the map $V(K) \to W(K)$ is a $\tau$-local homeomorphism.
  \item The polynomial inverse function theorem holds: if we have a
    polynomial map
    \begin{align*}
      \bar{f} : K^n &\to K^n \\
      (x_1,\ldots,x_n) &\mapsto (f_1(\bx),f_2(\bx),\ldots,f_n(\bx))
    \end{align*}
    whose Jacobian matrix $\{\partial f_i / \partial x_j\}_{i,j \le
      n}$ is invertible at a point $\ba$, then the map $\bar{f}$ is a
    $\tau$-local homeomorphism at $\ba$.
  \item The polynomial implicit function theorem holds: if we have a
    polynomial map
    \begin{align*}
      \bar{f} : K^n \times K^m &\to K^m \\
      (x_1,\ldots,x_n;y_1,\ldots,y_m) &\mapsto (f_1(\bx,\by),\ldots,f_m(\bx,\by)),
    \end{align*}
    and the Jacobian matrix $\{\partial f_i/\partial y_j\}_{i,j \le m}$ is invertible at a point $(\ba,\bb) \in K^n \times K^m$, then there are $\tau$-open neighborhoods $U \ni \ba$ and $V \ni \bb$ and a $\tau$-continuous function $g : U \to V$ such that
    \begin{equation*}
      f(\ba,\bb) = \bar{0} \iff g(\ba) = \bb, \text{ for any $\ba \in U$ and $\bb \in V$.}
    \end{equation*}
  \end{enumerate}
\end{proposition}
\begin{proof}
	Dittmann, Walsberg, and Ye have done all the heavy lifting in \cite{hensquot2}:
  \begin{itemize}
  \item $(1)\implies(2)$: Note that if $X, Y$ are topological spaces,
    then a continuous, open map $X \to Y$ is a local homeomorphism if
    and only if $X \to Y$ is locally injective, or equivalently, the
    diagonal $X \to X \times_Y X$ is an open embedding.  In our case,
    the map $V(K) \to W(K)$ is an open map by
    \cite[Proposition~8.6(5)]{hensquot2}, and the diagonal map $V(K)
    \to V(K) \times_{W(K)} V(K)$ is an open embedding because $V \to V
    \times_W V$ is an open embedding of schemes by
    \cite[Lemma~02GE(1)]{stacks-project}, recalling that etale
    morphisms are unramified \cite[Lemma~02GK]{stacks-project}.
  \item $(2)\implies(3)$: Polynomial maps are etale at the points
    where the Jacobian matrix is invertible.  More precisely, if $V$
    is the open subscheme of $\Aa^n$ on which the Jacobian determinant
    doesn't vanish, then $\bar{f} : V \to \Aa^n$ is etale, so $V(K)
    \to K^n$ is a local homeomorphism by the previous point.
  \item$(3)\implies(4)$: The proof is standard---one applies the
    inverse function theorem to the map $(\bx,\by) \mapsto
    (\bx,f(\bx,\by))$.  We omit the details.
  \item $(4)\implies(1)$: The proof of
    \cite[Proposition~8.3]{hensquot2} shows that gt-henselianity follows from the
    \emph{1-variable} implicit function theorem for polynomials, i.e.,
    condition (4) in the case when $m=1$.  \qedhere
  \end{itemize}
\end{proof}
Here are a few other facts from \cite{hensquot2} which help motivate gt-henselianity:
\begin{fact} \phantomsection \label{MCF}
  \begin{enumerate}
  \item A field topology $\tau$ is t-henselian if and only if it is a
  gt-henselian V-topology \cite[Proposition~8.3]{hensquot2}.
\item   Let $K$ be a field, and $R$ be a proper subring such that $K =
  \Frac(R)$ and $R$ is a henselian local ring.  Then the $R$-adic
  topology is gt-henselian. \cite[Proposition~8.5]{hensquot2}
\item If $(K,\tau)$ is gt-henselian, then $K$ is large \cite[Corollary~8.15]{hensquot2}.
\item The class of gt-henselian fields is a ``local class'' in the
  sense of \cite{PZ}, i.e., defined by a class of ``local sentences''.
  This is clear from examining Definition~\ref{original-def}, as noted
  in the proof of \cite[Lemma~8.10]{hensquot2}.
  \end{enumerate}
\end{fact}
The following is \cite[Conjecture~1.2]{noetherian}, modulo \cite[Proposition~3.6]{noetherian}:
\begin{conjecture}[Generalized henselianity conjecture] \label{genhencon}
  If $R$ is an NIP integral domain, then $R$ is a henselian local
  ring.
\end{conjecture}
The GHC is known to hold when $R$ has finite dp-rank or positive
characteristic (Fact~\ref{fact1.7}).
\begin{proposition} \label{gt-hens-prop}
  Let $(K,+,\cdot,\ldots)$ be a NIP expansion of a field.  Suppose the
  GHC holds or $\characteristic(K) = p > 0$ or $\dpr(K) = n < \omega$.
  If $\tau$ is a definable field topology on $K$, then $\tau$ is
  gt-henselian and $K$ is large.
\end{proposition}
\begin{proof}
  Largeness follows from gt-henselianity by Fact~\ref{MCF}(3).  For
  gt-henselianity, we may assume that $K$ is highly saturated and
  $\tau$ is induced by an externally definable integral domain $R
  \subsetneq K$ with $\Frac(R) = K$, as in the proof of
  Theorem~\ref{wn-thm}.  We can expand $K$ by adding a unary predicate
  for $R$; this doesn't affect the dp-rank \cite[Fact~3.1]{dEHJ} or
  the characteristic.  Then $R$ is a henselian local ring by the GHC
  or its known cases.  Therefore $\tau$ is gt-henselian by
  Fact~\ref{MCF}(2).
\end{proof}

\section{Topological reformulation of the Generalized Henselianity Conjecture} \label{sec-reformulate}
In this section, we show that the Generalized Henselianity Conjecture
can be reformulated as a statement about definable topologies on NIP
fields (Theorem~\ref{topreform}).  We first extract a technical tool
from \cite[Section~3]{noetherian}.
\begin{definition}[{\cite[Definition~3.3]{noetherian}}]
  A \emph{problematic ring} is an integral domain $R$ with exactly two
  maximal ideals $\mm_1, \mm_2$ and with $R/\mm_i$ infinite for
  $i=1,2$.
\end{definition}
\begin{fact} \label{evil}
  If the generalized henselianity conjecture fails, then there is a
  problematic NIP ring.
\end{fact}
\begin{proof}
  As discussed in \cite{noetherian} above Definition~3.3, the class of
  NIP rings is a ``pre-henselizing class''
  \cite[Definition~3.2]{noetherian}.  If there are no problematic NIP
  rings, then the class is a ``henselizing class''
  \cite[Definition~3.3]{noetherian}, which implies that every NIP
  integral domain is a henselian local ring
  \cite[Proposition~3.4]{noetherian}, i.e., the GHC holds.
\end{proof}
Recall our convention that ``ring topologies'' and ``field
topologies'' are non-discrete and Hausdorff.
\begin{remark} \label{plus-rem}
  Let $K$ be a field and let $\tau_1$ and $\tau_2$ be two field
  topologies on $K$.  Let $\tau_1 + \tau_2$ be the topology on $K$
  generated by $\tau_1$ and $\tau_2$.
  \begin{enumerate}
  \item $\tau_1 + \tau_2$ is a field topology on $K$, or discrete.
  \item Let $\mathcal{B}_i$ be a basis for $\tau_i$, for $i = 1, 2$.
    Then $\{U \cap V : U \in \mathcal{B}_1, ~ V \in \mathcal{B}_2\}$
    is a basis for $\tau_1 + \tau_2$.
  \item If $\tau_1$ and $\tau_2$ are definable in some structure
    $(K,\ldots)$, then so is $\tau_1 + \tau_2$.
  \end{enumerate}
  The proofs are straightforward.
\end{remark}
Say that two topologies $\tau_1, \tau_2$ on $K$ are \emph{independent} if $U \cap V \ne \varnothing$ for any non-empty $\tau_1$-open set $U \subseteq K$ and non-empty $\tau_2$-open set $V \subseteq K$.
\begin{lemma} \label{plus-lem}
  Let $\tau_1$ and $\tau_2$ be independent field topologies on $K$.
  \begin{enumerate}
  \item $\tau_1 + \tau_2$ is non-discrete, and is therefore a field
    topology.
  \item $\tau_1 + \tau_2$ is not gt-henselian.
  \end{enumerate}
\end{lemma}
\begin{proof}
  \begin{enumerate}
  \item Let $\Delta : K \to K \times K$ be the diagonal embedding.
    Then $(K,\tau_1+\tau_2) \to (K,\tau_1) \times (K,\tau_2)$ is a
    topological embedding (the topology $\tau_1 + \tau_2$ is precisely
    the pullback of the product topology $\tau_1 \times \tau_2$ along
    $\Delta$).  If $\tau_1 + \tau_2$ is the discrete topology on $K$,
    then the diagonal $\Delta(K) \subseteq K \times K$ is discrete as
    a subset of $(K,\tau_1) \times (K,\tau_2)$.  On the other hand,
    the independence of $\tau_1$ and $\tau_2$ means precisely that
    $\Delta(K)$ is dense in $(K,\tau_1) \times (K,\tau_2)$.  The
    product space $(K,\tau_1) \times (K,\tau_2)$ is a non-empty
    Hausdorff space with no isolated points, so no discrete set is
    dense.
  \item Let $\tau_3 = \tau_1 + \tau_2$.  Suppose for the sake of
    contradiction that $\tau_3$ is gt-henselian.  Let $f(x) = x^2 -
    x$.  Then $f'(0) = -1 \ne 0$, so $f$ is a $\tau_3$-local
    homeomorphism at $0$ by Proposition~\ref{understand}(3).

    For $i=1,2$, take $U_i$ a $\tau_i$-neighborhood of 0, small enough
    that $U_i \cap (1-U_i) = \varnothing$.  Then $U_1 \cap U_2$ is a
    $\tau_3$-neighborhood of 0.  Since $f$ is a $\tau_3$-local
    homeomorphism at 0, the image $f(U_1 \cap U_2)$ is a
    $\tau_3$-neighborhood of $f(0) = 0$.  Thus there are
    $\tau_i$-neighborhoods $V_i \ni 0$ for $i = 1, 2$ such that $f(U_1
    \cap U_2) \supseteq V_1 \cap V_2$.  By $\tau_i$-continuity of $f$,
    there are $\tau_i$-neighborhoods $U'_i \subseteq U_i$ such that
    $f(U'_i) \subseteq V_i$.  Note that
    \begin{gather*}
      f(x) = f(1-x) \\
      f(x)=f(y) \iff x \in \{y,1-y\}.
    \end{gather*}
    By independence of $\tau_1$ and $\tau_2$, there is some $a \in
    (1 - U'_1) \cap U'_2$.  Then $f(a) = f(1-a) \in f(U'_1) \subseteq V_1$, and
    $f(a) \in f(U'_2) \subseteq V_2$.  Thus $f(a) \in V_1
    \cap V_2 \subseteq f(U_1 \cap U_2)$.  Take $b \in U_1 \cap U_2$
    with $f(b) = f(a)$.  Then $b = a$ or $b = 1-a$.  The first case
    cannot happen because $b \in U_1$ and $a \in 1 - U_1$.  The
    second case cannot happen because $b \in U_2$ and $1-a \in 1 -
    U_2$.  \qedhere
  \end{enumerate}
\end{proof}

\begin{corollary} \label{redundant-cor}
  Let $K$ be a NIP field, possibly with extra structure.  Let $\tau_1,
  \tau_2$ be two definable topologies on $K$.  If $\characteristic(K)
  = p > 0$ or $\dpr(K) = n < \omega$ or the GHC holds, then $\tau_1$
  and $\tau_2$ are \emph{not} independent.
\end{corollary}
\begin{proof}
  Suppose not.  By Remark~\ref{plus-rem} and Lemma~\ref{plus-lem}, the
  topology $\tau_1 + \tau_2$ is a definable field topology on $K$ and
  $\tau_1 + \tau_2$ is \emph{not} gt-henselian.  This contradicts
  Proposition~\ref{gt-hens-prop}.
\end{proof}

\begin{remark} \label{remdependent}
  Let $\tau_1$ and $\tau_2$ be two field topologies on $K$.  Then
$\tau_1$ and $\tau_2$ are independent if and only if
\begin{equation*}
  (a + U) \cap (b - V) \stackrel{?}{\ne} \varnothing
\end{equation*}
for any $\tau_1$-neighborhood $U \ni 0$ and $\tau_2$-neighborhood $V
\ni 0$.  Applying an affine transformation, it suffices to consider
the case where $a = 0$ and $b = 1$:
\begin{equation*}
  U \cap (1 - V) \stackrel{?}{\ne} \varnothing.
\end{equation*}
Equivalently, it suffices to check that
\begin{equation*}
  1 \stackrel{?}{\in} U + V
\end{equation*}
for any $\tau_1$-neighborhood $U \ni 0$ and $\tau_2$-neighborhood $V
\ni 0$.  In other words,
\begin{quote}
  $\tau_1$ and $\tau_2$ are independent if and only if $1 \in U + V$
  for every $\tau_1$-neighborhood $U \ni 0$ and $\tau_2$-neighborhood
  $V \ni 0$.
\end{quote}
\end{remark}
\begin{lemma} \label{independence-trick}
  Let $A$ be an integral domain with exactly two maximal ideals $\pp$
  and $\qq$.  Suppose that $(A \setminus \pp) \cdot (A \setminus \qq)
  = A \setminus \{0\}$.  Then the two topologies on $K = \Frac(A)$
  induced by $A_\pp$ and $A_\qq$ are independent.
\end{lemma}
\begin{proof}
  First we show that $A_\pp + A_\qq = K$.  Any element of $K$ can be
  written as $a/b$ with $a \in A$ and $b \in A \setminus \{0\}$.  By
  the assumption on $A$, $b = st$ for some $s \in A \setminus \pp$ and
  $t \in A \setminus \qq$.  Neither maximal ideal of $A$ contains
  $\{s,t\}$, so the ideal $\langle s , t \rangle \subseteq A$ is
  improper: $\langle s, t \rangle = A$.  Therefore $a = sy + tx$ for
  some $x,y \in A$.  Then
  \begin{equation*}
    \frac{a}{b} = \frac{tx + sy}{st} = \frac{x}{s} + \frac{y}{t} \in
    A_\pp + A_\qq.
  \end{equation*}
  This proves the claim that $A_\pp + A_\qq = K$.

  Next, we show that the topologies on $K$ induced by $A_\pp$ and
  $A_\qq$ are independent.  By the discussion before the Lemma, it
  suffices to show that
  \begin{equation*}
    1 \stackrel{?}{\in} \alpha A_\pp + \beta A_\qq
  \end{equation*}
  for $\alpha, \beta \in K^\times$.  The ring topology on $K$ induced by
  $A$ is non-discrete, so there is some non-zero $\gamma \in \alpha A
  \cap \beta A$.  Then $\gamma/\alpha \in A \subseteq A_\pp$, so
  $\alpha A_\pp \supseteq \gamma A_\pp$.  Similarly $\beta A_\qq
  \supseteq \gamma A_\qq$.  Then
  \begin{equation*}
    \alpha A_\pp + \beta A_\qq \supseteq \gamma A_\pp + \gamma A_\qq =
    \gamma K = K \ni 1.  \qedhere
  \end{equation*}
\end{proof}

\begin{proposition} \label{third-step}
  If the GHC fails, then there is an NIP expansion of a field $K$ with
  two independent definable field topologies $\tau_1$ and $\tau_2$.
  Moreover, we can arrange that $\tau_1$ and $\tau_2$ are induced by
  local rings on $K$.
\end{proposition}
\begin{proof}
  Suppose the GHC fails.  By Fact~\ref{evil}, there is an NIP integral domain
  $A$ with exactly two maximal ideals $\pp$ and $\qq$.  Let $S =
  (A \setminus \pp) \cdot (A \setminus \qq)$.  Then $S$ is a
  multiplicative subset of $A \setminus \{0\}$.  Let $\rr$ be maximal
  among ideals of $A$ disjoint from $S$.  By a well-known fact in commutative algebra \cite[Proposition~2.11]{eisenbud}, $\rr$
  is a prime ideal.  Moreover, if $x \in A \setminus \rr$, then $\rr
  + \langle x \rangle$ intersects $S$, meaning that $x \equiv y
  \pmod{\rr}$ for some $y \in S$.  By definition of $S$, we have shown
  the following:
  \begin{claim} \label{cl77}
    If $x \in A \setminus \rr$, then there are $s \in A \setminus
    \pp$ and $t \in A \setminus \qq$ such that $x \equiv st
    \pmod{\rr}$.
  \end{claim}
  Note that $\rr$ is disjoint from $A \setminus \pp$ and $A
  \setminus \qq$, so $\rr \subseteq \pp \cap \qq$.

  Let $A' = A/\rr$.  The maximal ideals of $A'$ are $\pp' = \pp/\rr$
  and $\qq' = \qq/\rr$.  By the Claim, any non-zero element of $A'$ is
  equal to a product of an element of $A' \setminus \pp'$ and an
  element of $A' \setminus \qq'$.  By Lemma~\ref{independence-trick},
  the two localizations $A'_{\pp'}$ and $A'_{\qq'}$ induce independent
  topologies on $K = \Frac(A')$.

  The ring $A'$ is NIP, because it is the quotient of the NIP ring $A$
  by the ideal $\rr$, which is externally definable by
  \cite[Proposition~2.13]{fpcase}.  Then the pair $(K,A')$ is NIP.  In
  this pair, the two primes $\pp'$ and $\qq'$ are definable
  \cite[Corollary~2.4]{fpcase}, and the sets $A'_{\pp'}$ and
  $A'_{\qq'}$ are definable.  Therefore, the two induced topologies on
  $K$ are definable in $(K,A')$, an NIP expansion of a field.
\end{proof}

\begin{theorem} \label{topreform}
  The following statements are equivalent:
  \begin{enumerate}
  \item \label{green} The GHC holds: if $R$ is an NIP integral domain, then $R$ is a
    Henselian local ring.
  \item If $K$ is an NIP expansion of a field, and $\tau$ is a
    definable field topology on $K$, then $(K,\tau)$ is gt-henselian.
  \item \label{re4.5} If $K$ is an NIP expansion of a field, and $\tau_1, \tau_2$
    are two definable field topologies on $K$, then $\tau_1$ and
    $\tau_2$ are \emph{not} independent.
  \end{enumerate}
\end{theorem}
\begin{proof}
  (1)$\implies$(2) is Proposition~\ref{gt-hens-prop}.
  (2)$\implies$(3) is by Remark~\ref{plus-rem} and
  Lemma~\ref{plus-lem}, just as in the proof of
  Corollary~\ref{redundant-cor}.  (3)$\implies$(1) is
  Proposition~\ref{third-step}.
\end{proof}
Condition (\ref{re4.5}) in Theorem \ref{topreform} is somewhat analogous
to the following:
\begin{fact} \label{another-fact}
  The following are equivalent:
  \begin{enumerate}
  \item The henselianity conjecture for NIP valuation rings holds: any
    NIP valuation ring is henselian.
  \item No NIP expansion of a field admits two independent definable
    V-topologies.
  \end{enumerate}
\end{fact}
This is essentially \cite[Corollary~5.6]{hhj-v-top}, modulo
\cite[Proposition~3.5]{hhj-v-top}.
\begin{remark}
  The proof of Theorem~\ref{topreform} is easier than
  Fact~\ref{another-fact} because we can use the field topology
  $\tau_1 + \tau_2$ to directly break gt-henselianity.  In a sense we
  are cheating by considering non-local rings in the generalized
  henselianity conjecture.  More precisely, Condition~(\ref{green}) of
  Theorem~\ref{topreform} doesn't assume $R$ is a local ring---we are
  smuggling in the statement ``NIP integral domains are local'' into
  our conjecture.

  In the next section, we will resolve this discrepancy.
\end{remark}

\section{Reduction to local rings}
In this section, we show that the Generalized Henselianity Conjecture
can be checked on local rings (Theorem~\ref{local-form}).  The main
difficulty will be proving Proposition~\ref{chaotic}, which says that
an NIP field cannot admit two independent definable gt-henselian
topologies.  The proof strategy is similar to
\cite[Section~4]{hhj-v-top}, and uses some algebraic geometry.

In the following, an ``algebraic curve'' over $K$ means a
geometrically integral 1-dimensional smooth quasiprojective variety
over $K$.  We begin with a variant of \cite[Lemma~2.11]{hhj-v-top}.
\begin{lemma} \label{nonisol}
  Let $K$ be a field and $\tau$ be a gt-henselian field topology on
  $K$.  Let $C$ be an algebraic curve over $K$.  Then the
  topology on $C(K)$ induced by $\tau$ has no isolated points.
  Equivalently, every non-empty $\tau$-open set in $C(K)$ is infinite.
\end{lemma}
\begin{proof}
  For any point $p \in C(K)$, there is an open subscheme $U \subseteq
  C$ with $p \in U(K)$ and an etale morphism $f : U \to \Aa^1$.  See
  \cite[Theorem~039Q]{stacks-project} or the proof of
  \cite[Lemma~2.11]{hhj-v-top}.  By
  Proposition~\ref{understand}(\ref{etale}), the map $f : U(K) \to
  \Aa^1(K) = K$ is a local homeomorphism.  Since $K$ has no isolated
  points, neither does $U(K)$, and $p$ is non-isolated.
\end{proof}
Next we need a slight variant of \cite[Lemma~2.10]{hhj-v-top}.
\begin{lemma} \label{RR}
  Let $C$ be a genus $g$ projective algebraic curve over an infinite
  field $K$.  Let $p_1,\ldots,p_n,q$ be distinct points in $C(K)$.  If
  $n \ge 2g+1$, then there is a morphism $f : C \to \Pp^1$ of degree
  $n$, with:
  \begin{itemize}
  \item A simple pole at $p_i$, for each $i$.
  \item No other poles.
  \item A simple zero at $q$.
  \item $n-1$ other zeros (counting multiplicities).
  \end{itemize}
\end{lemma}
This is probably trivial to algebraic geometers, but here is the proof
for the rest of us:
\begin{proof}
  Consider the following divisors:
  \begin{align*}
    D &= p_1 + p_2 + \cdots + p_n - q \\
    D_0 &= p_1 + p_2 + \cdots + p_n - 2q \\
    D_i &= p_1 + \cdots + \widehat{p_i} + \cdots + p_n - q \text{ for } 1 \le i \le n.
  \end{align*}
  Each of these is a divisor of degree at least $2g-1$, so the
  Riemann-Roch theorem gives
  \begin{align*}
  	\dim H^0(D) &= \deg(D) - g + 1 = n - g \\
    \dim H^0(D) &= \deg(D_i) - g + 1 = n - g - 1 \text{ for } 0 \le i \le n,
  \end{align*}
where $H^0(D)$ is the vector space of meromorphic functions $f$ with $(f) + D \ge 0$.
  Each vector space $H^0(D_i)$ is a subspace of
  $H^0(D)$, as $D \ge D_i$.  By Fact~\ref{cover-fail} and the fact that $K$
  is infinite, $H^0(D) \supsetneq \bigcup_{i=0}^n
  H^0(D_i)$.  Then we can find a meromorphic function $f : C
  \to \Pp^1$ such that $f \in H^0(D) \setminus
  \bigcup_{i=0}^n H^0(D_i)$, meaning that the divisor $(f)$
  satisfies
  \begin{align*}
    (f) + D &\ge 0 \\
    (f) + D_i & \not \ge 0 \text{ for } 0 \le i \le n.
  \end{align*}
  That is,
  \begin{gather*}
    (f) + p_1 + \cdots + p_n - q \ge 0 \\
    (f) + p_1 + \cdots + p_n - 2q \not \ge 0 \\
    (f) + p_1 + \cdots + \widehat{p_i} + \cdots + p_n - q \not \ge 0
  \end{gather*}
  Then $f$ has a zero of order 1 at $q$, and a pole of order 1 at
  $p_i$ for $1 \le i \le n$.  By the first line, $f$ has no poles at
  any points besides $p_1,\ldots,p_n$.  It follows that the total
  number of poles, counting multiplicities, is $n$, so $f$ has degree
  $n$.  Then the number of zeros is $n$, counting multiplicities, so
  $f$ has $n-1$ other zeros besides $q$.
\end{proof}
Now we need the following variant of \cite[Proposition~4.2]{hhj-v-top}.
\begin{lemma} \label{approx-curves}
  Let $K$ be a field.  Let $\tau_1$ and $\tau_2$ be two independent
  gt-henselian field topologies on $K$.  Let $C$ be an algebraic curve
  over $K$.  Then the two topologies on $C(K)$ induced by $\tau_1$ and
  $\tau_2$ are independent.
\end{lemma}
\begin{proof}
  Note that $K$ is infinite, since it admits a non-discrete Hausdorff
  topology.  Replacing $C$ with its projective completion, we may
  assume $C$ is projective.  Let $U_i \subseteq C(K)$ be a non-empty
  $\tau_i$-open set for $i = 1, 2$.  We must show that $U_1 \cap U_2
  \ne \varnothing$.  Take $n \ge 2g+1$, where $g$ is the genus of $C$.
  By Lemma~\ref{nonisol}, $U_1$ is infinite.  Take distinct points
  $p_1,\ldots,p_n \in U_1$, and take $q \in U_2$.  If $q \in
  \{p_1,\ldots,p_n\}$, then $q \in U_1 \cap U_2$ and the proof is
  complete.  So we may assume $q \notin \{p_1,\ldots,p_n\}$.  By
  Lemma~\ref{RR}, there is a morphism of $K$-varieties $f : C \to
  \Pp^1$ such that
  \begin{itemize}
  \item $f$ has degree $n$.
  \item $f$ has a simple pole at $p_i$ for $i = 1, \ldots, n$, and no
    other poles.
  \item $f$ has a simple zero at $q$, and possibly other zeros
    elsewhere.
  \end{itemize}
  Then $f$ is etale at the points $p_1,\ldots,p_n,q$, because the
  poles and zeros are simple \cite[proof of Proposition~2.9]{hhj-v-top}.  By
  Proposition~\ref{understand}(\ref{etale}), $f : U(K) \to \Pp^1(K)$
  is a local homeomorphism at these points, with respect to either
  $\tau_1$ or $\tau_2$.  Focusing on $\tau_1$ and the $p_i$, we can
  find $\tau_1$-open neighborhoods $V_1,\ldots,V_n,V_\infty$:
  \begin{gather*}
    \infty \in V_\infty \subseteq \Pp^1(K) \\
    p_i \in V_i \subseteq C(K) \text{ for } 1 \le i \le n
  \end{gather*}
  such that
  \begin{itemize}
  \item For $i = 1, \ldots, n$, the function $f : C(K) \to \Pp^1(K)$
    maps $V_i$ homeomorphically to $V_\infty$ with respect to
    $\tau_1$.
  \item The sets $V_1,\ldots,V_n$ are pairwise disjoint.
  \item Each $V_i$ is a subset of $U_1$.
  \end{itemize}
  We can also find $\tau_2$-open neighborhoods $W_q, W_0$:
  \begin{gather*}
    0 \in W_0 \subseteq \Pp^1(K) \\
    q \in W_q \subseteq C(K) 
  \end{gather*}
  such that $f$ maps $W_q$ homeomorphically to $W_0$ with respect to
  $\tau_2$, and $W_q \subseteq U_2$.

  Since $\tau_1$ and $\tau_2$ are independent, we can find some point
  $a \in V_\infty \cap W_0 \subseteq \Pp^1(K)$.  For $i = 1,\ldots,n$,
  let $b_i$ be the unique point in $V_i$ mapping to $a \in V_\infty$.
  The $b_i$ are pairwise distinct because the $V_i$ are pairwise
  disjoint.  Then $\{b_1,\ldots,b_n\}$ is the entire fiber
  $f^{-1}(a)$, because $f$ has degree $n$.  Let $c$ be the unique
  point in $W_q$ mapping to $a$.  Then $c \in f^{-1}(a)$, so $c = b_i$
  for some $i$.  Finally,
  \begin{gather*}
    c \in W_q \subseteq U_2 \\
    c = b_i \in V_i \subseteq U_1,
  \end{gather*}
  so $c \in U_1 \cap U_2$.
\end{proof}
The proof of Lemma~\ref{approx-curves} could probably be simplified by
following the proof of \cite[Proposition~4.2]{hhj-v-top} more closely.
The proof given here is chosen to minimize the use of the assumption
``$\tau_2$ is gt-henselian'', in the hopes that the proof could be
modified to prove Conjecture~\ref{un1} in the appendix.

Next, we need an analog of \cite[Theorem~5.3]{hhj-v-top}.
\begin{proposition} \label{chaotic}
  Let $K$ be a field, possibly with extra structure.  Let $\tau_1$ and
  $\tau_2$ be two independent, definable gt-henselian field topologies
  on $K$.  Then $K$ has the independence property.
\end{proposition}
\begin{proof}
  If $\characteristic(K) = p > 0$, then $K$ has the independence property by
  Corollary~\ref{redundant-cor}.  So we may assume $\characteristic(K) = 0$.  Take definable $\tau_1$-open
  neighborhoods $U, V \ni 1$ such that the squaring map is a bijection
  from $U$ to $V$.  Shrinking $U$ and $V$, we can assume $U \cap (-U) =
  \varnothing$.  If $x \in V$, let $\sqrt{x}$ denote the unique square
  root in $U$.  Similarly, let $U', V' \ni 1$ be definable
  $\tau_2$-open neighborhoods such that the squaring map is a
  bijection from $U'$ to $V'$ and $U' \cap (-U') = \varnothing$.  If $x
  \in V \cap V'$, then $x$ has a unique square root $\sqrt{x}$ in $U$,
  as well as a unique square root in $U'$, which will be either
  $\sqrt{x}$ or $-\sqrt{x}$.  Let
  \begin{gather*}
    X^+ = \{x \in V \cap V' : \sqrt{x} \in U'\} \\
    X^- = \{x \in V \cap V' : -\sqrt{x} \in U'\}.
  \end{gather*}
  The two sets $X^+$ and $X^-$ are definable.  The set $V \cap V'$ is
  a disjoint union $X^+ \sqcup X^-$.
  \begin{claim} \label{cl83}
    For any distinct $c_1, \ldots, c_n \in K^\times$ and $S \subseteq
    \{1,\ldots,n\}$, there is some $a \in K^\times$ such that
    \begin{gather*}
      1 + c_i a \in X^+ \text{ if } i \in S \\
      1 + c_i a \in X^- \text{ if } i \notin S
    \end{gather*}
    for all $i = 1, 2, \ldots, n$.
  \end{claim}
  The claim implies that the condition $1 + xy \in X^+$ has the
  independence property.  It remains to prove the claim.

  We need elements $a, b_1, \ldots, b_n$ such that for each $i$,
  \begin{gather*}
    b_i^2 = 1 + ac_i \\
    b_i \in U \\
    b_i \in U' \text{ if } i \in S \\
    -b_i \in U' \text{ if } i \notin S.
  \end{gather*}
  (Indeed, $b_i \in U \implies b_i^2 = 1 + ac_i \in V$, and similarly
  $\pm b_i \in U' \implies 1 + ac_i \in V'$.  Then $1 + ac_i \in V
  \cap V' = X^+ \sqcup X^-$.  Also, $b_i = \sqrt{1 + ac_i}$, and so $i
  \in S \implies 1 + ac_i \in X^+$.  Similarly, $i \notin S \implies 1
  + ac_i \in X^-$.)

  Let $C$ be the algebraic curve in the variables $(x,y_1,\ldots,y_n)$
  defined by the equations
  \begin{equation*}
    y_i^2 = 1 + xc_i.
  \end{equation*}
  Then $C$ is smooth and irreducible, using the fact that the $c_i$
  are distinct \cite[Chapter 10, Exercise 4]{field-arithmetic}.  Let $p = (0,1,1,1,\ldots,1)$.
  Let $q = (0,s_1,\ldots,s_n)$, where $s_i$ is $+1$ if $i \in S$ and
  $-1$ if $i \notin S$.  Then $p$ and $q$ are both points on $C$.  By
  Lemma~\ref{approx-curves}, we can find a point $(a,b_1,\ldots,b_n)$
  on the curve $C$ which is arbitrarily $\tau_1$-close to $p$ and
  $\tau_2$-close to $q$.  Then $b_i^2 = 1 + ac_i$ for each $i$, and
  $b_i \in U$ (since $1 \in U$) and $\pm b_i \in U'$ (because $\pm s_i
  \in U'$), where the sign of the $\pm$ depends on whether $i \in S$.  This proves the claim.
\end{proof}

\begin{theorem} \label{local-form}
  The following statements are equivalent:
  \begin{enumerate}
  \item If $R$ is an NIP integral domain, then $R$ is a henselian local
    ring.
  \item If $R$ is an NIP ring, then $R$ is a direct product of
    finitely many henselian local rings.
  \item If $R$ is an NIP local ring, then $R$ is henselian.
  \item If $R$ is an NIP local domain, then $R$ is henselian.
  \end{enumerate}
\end{theorem}
\begin{proof}
  (1)$\iff$(2) is \cite[Proposition~3.6]{noetherian}.  (2)$\implies$(3)$\implies$(4) is clear.  It
  remains to prove (4)$\implies$(1).  Suppose (1) fails.  By
  Proposition~\ref{third-step}, there is an NIP structure in which
  there is a definable field $K$ and definable local subrings $A_1,
  A_2 \subseteq K$ with $\Frac(A_i) = K$, such that if $\tau_i$ is the
  topology on $K$ induced by $A_i$, then $\tau_1$ and $\tau_2$ are
  independent.  By (4), $A_1$ and $A_2$ are henselian local rings,
  implying that $\tau_1$ and $\tau_2$ are gt-henselian.  Then $K$ has
  two definable, independent, gt-henselian field topologies.  By
  Proposition~\ref{chaotic}, $K$ has the IP, a contradiction.
\end{proof}

\section{Does the Shelah conjecture imply the Generalized Henselianity Conjecture?}
\label{thoughts}

Recall the three conjectures:
\begin{conjecture-star}[Shelah conjecture]
  If $K$ is an NIP field, then $K$ is finite or separably closed or
  real closed or $K$ admits a definable henselian valuation.
\end{conjecture-star}
\begin{conjecture-star}[Henselianity conjecture]
  If $(K,v)$ is an NIP valued field, then $v$ is henselian.
\end{conjecture-star}
\begin{conjecture-star}[Generalized henselianity conjecture]
  If $R$ is an NIP integral domain, then $R$ is a henselian local
  ring.
\end{conjecture-star}
Halevi, Hasson, and Jahnke prove that the Shelah conjecture
(Conjecture~\ref{shelconj}) implies the henselianity conjecture
(\ref{hensconj}).  Can we do the same with the \emph{generalized}
henselianity conjecture?
\begin{goal} \label{future-goal}
  Prove the generalized henselianity conjecture, \emph{assuming} the
  Shelah conjecture.
\end{goal}
Aside from putting the generalized henselianity conjecture on a firmer
foundation, this would reduce many questions about NIP rings and
topological fields to questions about NIP fields.

While I have not succeeded at Goal~\ref{future-goal}, I believe it may
be within reach of current techniques.  The obvious strategy is to
mimic the proofs in \cite{hhj-v-top}.  A key technical lemma there is
the following fact \cite[Proposition~4.2]{hhj-v-top}:
\begin{enumerate}
\item Let $C$ be an algebraic curve over a field $K$.  Let $\tau_1$
  and $\tau_2$ be two independent V-topologies on $K$, with $\tau_1$
  t-henselian.  Then $\tau_1$ and $\tau_2$ induce independent
  topologies on $C(K)$.
\end{enumerate}
Here, two topologies $\tau_1, \tau_2$ on a set $S$ are
\emph{independent} if any non-empty $\tau_1$-open set $U \subseteq S$
intersects any non-empty $\tau_2$-open set $V \subseteq S$.

Note that \cite[Proposition~4.2]{hhj-v-top} is a direct inspiration
for our Lemma~\ref{approx-curves}, which said:
\begin{enumerate}[resume]
\item Let $C$ be an algebraic curve over a field $K$.  Let $\tau_1$
  and $\tau_2$ be two independent field topologies on $K$, with $\tau_1$
  and $\tau_2$ gt-henselian.  Then $\tau_1$ and $\tau_2$ induce independent
  topologies on $C(K)$.
\end{enumerate}
In order to run the Halevi-Hasson-Jahnke arguments, we need something
stronger---a hybrid of \cite[Proposition~4.2]{hhj-v-top} and
Lemma~\ref{approx-curves}:
\begin{unlikely} \label{un1}
  Let $C$ be an algebraic curve over a field $K$.  Let $\tau_1$ and
  $\tau_2$ be two independent field topologies on $K$, with $\tau_1$
  t-henselian.  Then $\tau_1$ and $\tau_2$ induce independent
  topologies on $C(K)$.
\end{unlikely}
\begin{unlikely} \label{un2}
  Let $C$ be an algebraic curve over an algebraically closed field
  $K$.  Let $\tau_1$ and $\tau_2$ be two independent field topologies
  on $K$.  Then $\tau_1$ and $\tau_2$ induce independent topologies on
  $C(K)$.
\end{unlikely}
\begin{theorem} \label{roundabout}
  If the Shelah conjecture and Conjectures~\ref{un1} and \ref{un2}
  hold, then the Generalized Henselianity Conjecture holds.
\end{theorem}
\begin{proof}
  First note that Conjectures~\ref{un1} and \ref{un2} can be upgraded
  to the following statements, using the proof of
  Proposition~\ref{chaotic}:
  \begin{claim} \label{claim1}
    Let $K$ be a field, possibly with extra structure.  Let $\tau_1$
    and $\tau_2$ be two independent, definable field topologies on
    $K$, with $\tau_1$ t-henselian.  Then $K$ has the independence
    property.
  \end{claim}
  \begin{claim} \label{claim2}
    Let $K$ be an algebraically closed field, possibly with extra
    structure.  Let $\tau_1$ and $\tau_2$ be two independent,
    definable field topologies on $K$.  Then $K$ has the independence
    property.
  \end{claim}
  If the GHC fails, then there is an NIP expansion of a field $K$
  defining two independent field topologies $\tau_1$ and $\tau_2$, by
  Theorem~\ref{topreform}(\ref{re4.5}).  Then $\characteristic(K) = 0$
  by Corollary~\ref{redundant-cor}.  By the Shelah conjecture, one of
  the following cases holds:
  \begin{enumerate}
  \item \emph{$K$ is real closed or admits a definable henselian valuation
    $v$}.  Either way, $K$ admits a definable t-henselian topology
    $\tau_0$.  For $i = 1, 2$, the topology $\tau_i$ cannot be
    independent from $\tau_0$ by Claim~\ref{claim1}, and so $\tau_i
    \supseteq \tau_0$ by Fact~\ref{whoah} below.  Then $\tau_1$ and
    $\tau_2$ cannot be independent: if $U$ and $V$ are two disjoint
    non-empty $\tau_0$-open sets, then $U$ is a non-empty
    $\tau_1$-open set and $V$ is a non-empty $\tau_2$-open set, and
    yet $U \cap V = \varnothing$, contradicting independence.
  \item \emph{$K$ is algebraically closed}.  Then Claim~\ref{claim2}
    shows that $\tau_1$ and $\tau_2$ cannot be independent, a
    contradiction.  \qedhere
  \end{enumerate}
\end{proof}
\begin{fact} \label{whoah}
  Let $\tau_0$ and $\tau_1$ be two field topologies on $K$, with
  $\tau_0$ a V-topology.  Then $\tau_1 \supseteq \tau_0$ or the two
  topologies $\tau_1$ and $\tau_0$ are independent.
\end{fact}
Presumably, this is a well-known fact about topological fields.  Here is a proof:
\begin{proof}
  Consider $K$ in a language which makes all subsets of $K$ be
  definable.  Let $K^*$ be a highly saturated elementary extension of
  $K$.  For $X \subseteq K$, let $X^* \subseteq K^*$ be the
  corresonding subset of $K^*$.  Let
  \begin{align*}
    \mm &= \bigcap \{X^* : X \subseteq K, \text{ $X$ is a
      $\tau_1$-neighborhood of 0}\}
    \\
    \Oo &= \bigcup \{X^* : X \subseteq K, \text{ $X$ is $\tau_1$-bounded}\}
    \\
    I &= \bigcap \{X^* : X \subseteq K, \text{ $X$ is a
      $\tau_2$-neighborhood of 0}\}
  \end{align*}
  Thus $\mm$ and $I$ are the ``infinitesimals'' for $\tau_1$ and
  $\tau_2$, while $\Oo$ is the ring of ``bounded'' elements.  Note
  that the intersections and union are filtered.  The fact that
  $\tau_1$ is a V-topology implies that $\Oo$ is a non-trivial
  valuation ring with maximal ideal $\mm$ (see
  \cite[Lemma~2.6]{prdf6}, for example).  The fact that $\tau_2$ is a
  field topology implies
  \begin{equation*}
    x \in I \implies \frac{x}{1+x} \in I, \tag{$\ast$}
  \end{equation*}
  essentially because the map $x \mapsto x/(1+x)$ is continuous at 0.
  \begin{description}
  \item[Case 1:] $I \subseteq \Oo$.  In this case, we will prove
    $\tau_2 \supseteq \tau_1$.  Take a $\tau_1$-neighborhood $X \ni
    0$.  We must show that $X$ is a $\tau_2$-neighborhood of 0.  Take
    some non-zero $\epsilon \in \mm$.  Then
    \begin{equation*}
      \epsilon I \subseteq \epsilon \Oo \subseteq \mm \subseteq X^*.
    \end{equation*}
    By saturation, there is a $\tau_2$-neighborhood $Y \ni 0$ with
    $\epsilon Y^* \subseteq X^*$.  As $K \preceq K^*$, there is $a \in
    K^\times$ such that $aY^* \subseteq X^*$ and $aY \subseteq X$.
    Then $aY$ and $X$ are $\tau_2$-neighborhoods of 0.
  \item[Case 2:] $I \not\subseteq \Oo$.  Take $x \in I \setminus \Oo$.
    Since $\Oo$ is a valuation ring, $1/(1+x) \in \mm$.  On the other
    hand, $x/(1+x) \in I$ by ($\ast$).  Then
    \begin{equation*}
      1 = \frac{1}{1+x} + \frac{x}{1+x} \in \mm + I.
    \end{equation*}
    It follows that for any $\tau_1$-neighborhood $X \ni 0$ and
    $\tau_2$-neighborhood $Y \ni 0$ we have $1 \in X^* + Y^*$ and $1
    \in X + Y$.  By Remark~\ref{remdependent}, $\tau_1$ and $\tau_2$
    are independent.  \qedhere
  \end{description}
\end{proof}
To accomplish Goal~\ref{future-goal}, it remains to prove
Conjectures~\ref{un1} and \ref{un2}.  Unfortunately, these conjectures
are probably false as stated.  However, we could weaken the
conjectures by putting additional assumptions on $\tau_1$ and $\tau_2$
such as local boundedness.  This would still suffice for the proof of
Theorem~\ref{roundabout}, because definable field topologies on NIP
fields are locally bounded (Theorem~\ref{loc-bdd-thm2}).

More generally, if $P$ is any property known to hold for definable
field topologies on NIP fields, then we could restrict
Conjectures~\ref{un1} and \ref{un2} by assuming that $\tau_1$ and
$\tau_2$ satisfy property $P$.  Aside from local boundedness, here is
another such property:
\begin{definition}
  A topological field $(K,\tau)$ is ``$n$-semilocal'' if $\tau$ is locally bounded and the following condition holds:
  \begin{itemize}
  \item There are $a_0, a_1, \ldots, a_n \in K$ such
  that for every bounded set $B \subseteq K$ there is a bounded set $C
  \subseteq K$ such that for any $x \in B$ there is some $i \le n$
  such that $(x-a_i)^{-1} \in C$.
  \end{itemize} Moreover, $(K,\tau)$ is
  ``semilocal'' if it is $n$-semilocal for some $n$.
\end{definition}
Using the methods of \cite[Section~2.1]{prdf6}, one can prove the
following:
\begin{enumerate}
\item The class of $n$-semilocal topological fields is a local class
  in the sense of Prestel and Ziegler \cite{PZ}, i.e., defined by
  local sentences.
\item Suppose $\tau$ is a definable locally bounded field topology on
  $K$, $K^*$ is a highly saturated elementary extension of $K$, and $R
  \subseteq K^*$ is the ring of $\tau$-bounded elements (see
  \cite[Proposition~2.1]{prdf6}).  Then $\tau$ is $n$-semilocal if and
  only if $R$ has $n$ or fewer maximal ideals.\footnote{A key thing to observe is that if $R$ is a
$K$-algebra for some infinite field $K$, then the following three
conditions are equivalent:
\begin{itemize}
\item $R$ has at most $n$ maximal ideals.
\item For any distinct $a_0,\ldots,a_n \in K$ and any $x \in R$, there is
  some $i$ such that $x-a_i \in R^\times$.
\item There are $a_0,\ldots,a_n \in K$ such that for any $x \in R$, there is some $i$ such that $x-a_i \in R^\times$.
\end{itemize}}
\end{enumerate}
  If $K$ is an NIP expansion of a field and $\tau$ is a definable
  field topology on $K$, then $\tau$ is semilocal.  Indeed, if $R$ is
  the ring of $\tau$-bounded elements in some highly saturated
  elementary extension $K^* \succeq K$, then $R$ is externally
  definable, hence NIP, so it has finitely many maximal ideals by
  \cite[Corollary~2.3]{fpcase}.

  Therefore, it would suffice to prove Conjectures~\ref{un1} and
  \ref{un2} in the case where $\tau_1$ and $\tau_2$ are semilocal.

\begin{acknowledgment}
  The author was supported by the National Natural Science Foundation
  of China (Grant No.\@ 12101131) and the Ministry of Education of
  China (Grant No.\@ 22JJD110002).  Zhentao Zhang helped find a
  mistake in an earlier proof of Claim~\ref{enough-special}.
\end{acknowledgment}

\bibliographystyle{plain} \bibliography{mybib}{}

\end{document}